\documentclass[11pt,reqno]{amsart}
\usepackage{amsmath, amsfonts, amsthm, amssymb,amscd, graphicx, amscd}
\usepackage{float,epsf}
\usepackage[english]{babel}
\usepackage{enumerate}
\usepackage{tikz}
\usepackage{mathrsfs}
\usepackage[numbers,sort&compress]{natbib}

\usepackage{srcltx}
\usepackage{geometry}
\usepackage{verbatim}
\usepackage{mathrsfs}
\usepackage{hyperref}
\usepackage{enumitem} 
\usepackage{bbm} 
\usepackage{stmaryrd}

\usepackage{relsize}
\usepackage{exscale}

\usepackage{mathtools}

\newtheorem{thm}{Theorem}[section]

\newtheorem{lem}[thm]{Lemma}

\newtheorem{rem}[thm]{Remark}

\numberwithin{equation}{section}

\def\dd{{\rm d}}
\hypersetup{bookmarksdepth=2}

\begin{document}
\title{From the Maxwell-Schr\"odinger equation to the Euler-Maxwell equation in the semiclassical limit}
\author{Immanuel Ben-Porat}
\address{Immanuel Ben Porat, The University of Texas at Austin, Austin, TX 78712, USA.}
\email{immanuel.benporat@austin.utexas.edu}
\begin{abstract}
We derive the pressureless Euler-Maxwell system as a semiclassical limit from the Maxwell-Schr\"odinger equation. The method relies on studying  a quantum modulated energy accounting for the self-consistent magnetic effects. Crucially, a correction is incorporated in the  quantum modulated energy as to enable to close the Gr\"onwall estimate provided a smallness assumption on the initial data is imposed. This is the first result on the semiclassical limit of the Maxwell-Schr\"odinger equation avoiding any regularization or conditional uniform in $\hbar$ bounds on the wave function.  
\end{abstract}
\maketitle
\pagestyle{plain}
\section{Introduction}
\subsection{Background.} One of the most commonly studied models in non-relativistic quantum mechanics is the Hartree equation 
\begin{align}
 \begin{array}{lc}
 \begin{cases}
     i\hbar\partial_{t}\psi_{\hbar}=-\frac{\hbar^{2}}{2}\Delta \psi_{\hbar}+V\ast \left\vert \psi_{\hbar}\right\vert^{2}\psi_{\hbar}   \\
      \psi_{\hbar}(0,\cdot)=\psi_{\hbar}^{\mathrm{in}},& 
\end{cases}
 \end{array}
 \label{Hartree}
\end{align}
describing the evolution of a wave function $\psi_{\hbar}$ via a self-consistent quantum Hamiltonian given by means of an interaction $V:\mathbb{R}^
{3}\rightarrow\mathbb{R}$ and scaled via Planck's constant $0<\hbar\ll 1$. The problem of  identifying the classical dynamics resulting from Hartree's equation as Planck's constant $\hbar$ becomes negligible is a prototypical example of a  \textit{semiclassical limit}. The semiclassical limit is motivated by the fact that the scaling parameter $\hbar$ measures how quantum the system is, and as it becomes small the system is approximately classical. We refer to \cite{susskind2014quantum} for a particularly accessible exposition of the physical background. It turns out that the Vlasov equation is the kinetic PDE that is expected to be derived from \eqref{Hartree}, namely 
\begin{align}
\begin{array}{lc}
\begin{cases}
\partial_{t}f+\xi\cdot \nabla_{x}f-\nabla V\ast \rho_{f}\cdot \nabla_{\xi}f=0 , \ \rho_{f}=\int_{\mathbb{R}^{3}}f(t,x,\xi)\ \dd \xi \\
f(0,\cdot)=f^{\mathrm{in}}.  
\end{cases}
\end{array}
\label{Vlasov}
\end{align}
We do not aim to give an exhaustive account of the literature concerning the rigorous justification of the semiclassical limit leading from \eqref{Hartree} to \eqref{Vlasov} but refer to the foundational works by Lions-Paul \cite{lions1993mesures} and Markowich-Mauser \cite{MarkowichMauser1993}. Other more modern  works on the semiclassical limit include  \cite{golse2019empirical,golse2017schrodinger,lafleche2019propagation,lafleche2023strong, lafleche2021global,Smith2026HartreeVlasov}. We refer to Chapter 3 in \cite{maas2024optimal} for a detailed  review of quantum optimal transport and its applications in semiclassical and quantum many-body limits.  A more realistic (and hence more involved) description of quantum matter includes a quantum Hamiltonian with a self-consistent electromagnetic part, which leads to the Maxwell-Schr\"odinger equation given by
\begin{align}\tag{MS}
 \begin{array}{lc}
 \begin{cases}
      & i\hbar \partial_{t}\psi_{\hbar}=\frac{1}{2}(i\hbar \nabla +A_{\hbar})^{2}\psi_{\hbar}+\varphi_{\hbar}\psi_{\hbar}\\
      &-\Delta \varphi_{\hbar}=\rho_{\hbar} \\
      & (\partial_{tt}-\Delta)A_{\hbar}=\mathbb{P}J_{\hbar}\\  &\mathrm{div}A_{\hbar}=0 
 \end{cases}     
 \end{array} 
 \label{Maxwell Schrodinger eq}
\end{align}
supplemented with initial data 
\begin{align}\tag{MS$^\mathrm{in}$}
(\psi_{\hbar}(0,\cdot),A_{\hbar}(0,\cdot),\partial_{t}A_{\hbar}(0,\cdot))=(\psi^{\mathrm{in}}_{\hbar},A_{0,\hbar}^{\mathrm{in}},A^{\mathrm{in}}_{1,\hbar}),\ \left\Vert \psi^{\mathrm{in}}_{\hbar}\right\Vert_{2}=1,  \ \mathrm{div}(A_{0,\hbar}^{\mathrm{in}})=\mathrm{div}(A^{\mathrm{in}}_{1,\hbar})=0. \label{MS initial data}    
\end{align}
Here we designate by $\rho_{\hbar},J_{\hbar}$ the quantum density and quantum current, respectively,  defined by  
\begin{align}
J_{\hbar}\coloneqq-\Re(\overline{\psi_{\hbar}}(i\hbar \nabla+A_{\hbar})\psi_{\hbar}), \ \rho_{\hbar}\coloneqq\left\vert \psi_{\hbar}\right\vert^{2} \label{quantum density and current def}     
\end{align}
and we designate by $\mathbb{P}$ the Helmholtz-Leray projector defined by
\begin{align}
\mathbb{P}J\coloneqq 
(I-\nabla \Delta^{-1}\mathrm{div})J.\label{Leray projection def}    
\end{align}
Note that $A_{\hbar}:[0,T]\times \mathbb{R}^{3}\rightarrow \mathbb{R}^{3}$ is to be interpreted as a magnetic vector potential. The expected  semiclassical limit in this case is the non-relativistic Vlasov-Maxwell equation 
\begin{align}\tag{VM}
\begin{cases}
\begin{array}{lc}
\partial_{t}f+\xi\cdot \nabla_{x}f+(E+\xi\times B)\cdot \nabla_{\xi}f=0 \  \\
\partial_{t}E-\mathrm{curl}B=-J, \ J_{f}=\int_{\mathbb{R}^{3}}\xi f(t,x,\xi)\ \dd \xi\\
\partial_{t}B+\mathrm{curl}E=0\\
\mathrm{div}E=\rho_{f},\mathrm{div}B=0, \ \rho_{f}=\int_{\mathbb{R}^{3}}f(t,x,\xi)\ \dd \xi.  
\end{array}
\end{cases}
\label{Vlasov Maxwell}
\end{align}
The rigorous derivation leading from \eqref{Maxwell Schrodinger eq} to \eqref{Vlasov Maxwell} in the  limit as $\hbar \rightarrow 0$ is well recognized as an outstanding open problem. However, there have been several intermediate and related results which we now review. The semiclassical limit for a regularized version of \eqref{Maxwell Schrodinger eq} and \eqref{Vlasov Maxwell} (a regularization modeling extended charges) has been proved in \cite{leopold2026derivation}. The regularization has the clear advantage that it compensates for the loss of derivative which is usually a major source of difficulty in \eqref{Vlasov Maxwell}, and results from the fact that $E,B$ are evolve hyperbolically (unlike the electrostatic case demonstrated in \eqref{Hartree}).  In \cite{leopold2020derivation} the regularized Maxwell-Schr\"odinger equation is obtained as a mean-field limit from the Pauli-Fierz Hamiltonian and in \cite{Leopold2024MaxwellSchrodingerQED} the author goes further to derive the regularized Vlasov-Maxwell equation as a combined mean-field semiclassical limit from quantum many-body dynamics. It is also possible to study the semiclassical limit for \textit{monokinetic initial data} in the following manner. Introducing the monokinetic ansatz $f(t,x,\xi)=\rho(t,x)\delta(\xi-u(t,x))$ a direct calculation reveals that $f$ is a solution of \eqref{Vlasov Maxwell} iff $(\rho,u,B,E)$ is a solution of the following \textit{Euler-Maxwell system}: 
\begin{align}\tag{EM}
\begin{cases}
\begin{array}{lc}
\partial_{t}\rho+\mathrm{div}(\rho u)=0\\
\partial_{t}u+uD_{x}u=E+u\times B\\
\partial_{t}E-\mathrm{curl}B=-\rho u\\
\partial_{t}B+\mathrm{curl}E=0\\
\mathrm{div}E=\rho,\ \mathrm{div}(B)=0
\end{array}
\end{cases}
\label{Euler Maxwell}
\end{align}
supplemented with initial data 
\begin{align}\tag{EM$^\mathrm{in}$}
(\rho(0,\cdot),u(0,\cdot),B(0,\cdot),E(0,\cdot))=(\rho^{\mathrm{in}},u^{\mathrm{in}},B^{\mathrm{in}},E^{\mathrm{in}}),\ \rho^{\mathrm{in}}\in \mathcal{P}(\mathbb{R}^{3}),\ \mathrm{div}E^{\mathrm{in}}=\rho^{\mathrm{in}},\ \mathrm{div}B^{\mathrm{in}}=0.    \label{EMI} 
\end{align}
In this sense, \eqref{Euler Maxwell} may be viewed as a special case of \eqref{Vlasov Maxwell} with initial data $\rho^{\mathrm{in}}\delta(\xi-u^{\mathrm{in}})$.  The derivation of the Euler-Maxwell equation with pressure as a semiclassical limit has been studied in \cite{KimMoon2024} via the modulated energy method, which will be central for the present work as well. Nevertheless, the result in \cite{KimMoon2024} is conditional since it assumes a bound uniform in $\hbar$ on $\psi_{\hbar}$, i.e. a bound of the form 
\begin{align}
\underset{0<\hbar\leq 1}{\sup}\underset{t\in [0,T]}{\sup}\left\Vert \psi_{\hbar}(t,\cdot)\right\Vert_{\infty}<\infty. \label{uniform in h assumption}    
\end{align}
 Such bounds are not known at present, and in fact even in the simplified electrostatic settings of the Hartree equation \eqref{Hartree} an estimate of the form \eqref{uniform in h assumption} has been proved in $3D$ only in the case where $V$ is the screened Coulomb interaction and the initial data verifies smallness assumptions \cite{lafleche2021global}. The related problem of propagation of quantum velocity moments of arbitrary order  uniformly in $\hbar$ for \eqref{Hartree} is explicitly classified as an open problem in \cite{iacobelli2024enhanced}. However, the pressureless Euler-Poisson equation has  already been derived (non-conditionally) as a semiclassical limit in the important work by Golse-Paul \cite{golse2022mean}.  A more recent work which is particularly relevant to the problem at stake is \cite{salvi2025semi}, in which the author derives the  relativistic Euler-Maxwell equation from the massive Klein–Gordon–Maxwell system (which can be viewed as a relativistic version of \eqref{Maxwell Schrodinger eq}).   Other related works about the semiclassical limit of quantum dynamics coupled to Maxwell type equations include \cite{MauserMollerYang2025}, where the authors implement WKB methods to derive the Euler-Poisswell equation from the Pauli–Poisswell equation, and \cite{mauser2007convergence}, in which the combined non-relativistic semiclassical limit of the Dirac-Maxwell equation is proved. To finish this discussion, we also mention that it is interesting to study the semiclassical limit when the magnetic field is external rather than coupled -- this line of research was initiated  by Arnold-Steinr\"uck in  \cite{arnold1989electromagnetic}, who employed a Wigner transform approach, and was also undertaken more recently by the author in \cite{ben2024magnetic,porat2026magnetized}, by exploiting machinery from quantum optimal transport and quantum modulated energies.   
\subsection{Main results}
The main purpose of this work is to prove the semiclassical limit leading from \eqref{Maxwell Schrodinger eq} to \eqref{Euler Maxwell} for small initial data. This demonstrates a first step   towards the resolution of the formidable problem mentioned earlier regarding the derivation of \eqref{Vlasov Maxwell}  semiclassically from \eqref{Maxwell Schrodinger eq} -- our main result settles this problem for small monokinetic initial data. To the best of our knowledge, this is the first result regarding the semiclassical limit of the non-relativistic Maxwell-Schr\"odinger equation which does not necessitate any regularization or conditional assumptions of the type \eqref{uniform in h assumption}. Our underlying methodology rests upon identifying a  \textit{corrected quantum modulated energy}, which we now define. Given a solution $(\psi_{\hbar},A_{\hbar},\varphi_{\hbar})$ to \eqref{Maxwell Schrodinger eq} and a solution $(\rho,u,B,E)$ to \eqref{Euler Maxwell} set $E_{\hbar}\coloneqq-\partial_{t}A_{\hbar}-\nabla\varphi_{\hbar}$, $B_{\hbar}\coloneqq \mathrm{curl}A_{\hbar}$ and define the \textit{quantum modulated energy} by    
\begin{align}
\mathcal{H}_{\hbar}(t)\coloneqq \underset{\coloneqq \mathcal{U}_{\hbar}(t)}{\underbrace{\frac{1}{2}\int_{\mathbb{R}^{3}}\left\vert (i\hbar\nabla+A_{\hbar}+u)\psi_{\hbar}\right\vert^{2}\ \dd x}}+\underset{\coloneqq \mathcal{B}_{\hbar}(t)}{\underbrace{\frac{1}{2}\int_{\mathbb{R}^{3}}\left\vert B_{\hbar}-B\right\vert^{2}\ \dd x}} +\underset{\coloneqq \mathcal{E}_{\hbar}(t)}{\underbrace{\frac{1}{2}\int_{\mathbb{R}^{3}}\left\vert E_{\hbar}-E\right\vert^{2}\ \dd x} }\label{def of quantum modualted energy}  
\end{align}
and the \textit{corrected} quantum modulated energy by 
\begin{align}
\mathcal{G}_{\hbar}(t)\coloneqq \mathcal{H}_{\hbar}(t)-\int_{\mathbb{R}^{3}}u\cdot ((E_{\hbar}-E)\times (B_{\hbar}-B))\ \dd x. \label{corrected modulated energy}   \end{align}
Given a solution $\psi_{\hbar}$ to \eqref{Maxwell Schrodinger eq} we set $R_{\hbar}(t)\coloneqq\vert \psi_{\hbar}(t,\cdot)\rangle\vert \langle \psi_{\hbar}(t,\cdot)\vert$.  By direct calculation $R_{\hbar}$ is governed by the following equation: 
\begin{align}
 \begin{array}{lc}
 \begin{cases}
      & i\hbar \partial_{t}R_{\hbar}=\left[\frac{1}{2}(i\hbar \nabla +A_{\hbar})^{2}+\varphi_{\hbar},R_{\hbar}\right]  \\
      &-\Delta \varphi_{\hbar}=\rho_{\hbar}\\
      & (\partial_{tt}-\Delta)A_{\hbar}=\mathbb{P}J_{\hbar},\\  &\mathrm{div}A_{\hbar}=0. 
 \end{cases}     
 \end{array} 
\label{density formulation}
\end{align}
Here $\left[\cdot,\cdot \right]$ designates the commutator defined by $\left[T,S\right]\coloneqq TS-ST$. It will be useful to invoke the density operator  formulation in \eqref{density formulation} in some of the forthcoming calculations.
\begin{rem}
Note that the quantum modulated energy, defined in \eqref{def of quantum modualted energy}, may be equivalently written as 
\begin{align*}
\mathcal{H}_{\hbar}(t)&=\frac{1}{2}\mathrm{tr}\left((i\hbar\nabla+A_{\hbar}+u)^{2}R_{\hbar}\right)\\
&+\frac{1}{2}\int_{\mathbb{R}^{3}}\left\vert B_{\hbar}-B\right\vert^{2}\ \dd x+\frac{1}{2}\int_{\mathbb{R}^{3}}\left\vert E_{\hbar}-E\right\vert^{2}\ \dd x
.  
\end{align*}
In addition, it will be useful to use the duality definition of the quantum current $J_{\hbar}$, i.e. $J_{\hbar}$ is the unique vector-valued Radon measure such that for all $a\in W^{1,\infty}(\mathbb{R}^{3};\mathbb{R}^{3})$ it holds that  
\begin{align*}
\int_{\mathbb{R}^{3}}J_{\hbar}\cdot a\ \dd x=-\frac{1}{2}\mathrm{tr}\left((i\hbar\nabla+A_{\hbar})\vee a R_{\hbar}\right).     
\end{align*}
In the above formula and hereafter we denote by $\vee$ the anti-commutator defined by $$T\vee S\coloneqq TS+ST.$$ 
\end{rem}
Our main theorem is stated below: 
\begin{thm}
Suppose that: 
\begin{itemize}
    \item $\rho^{\mathrm{in}}\in H^{3}(\mathbb{R}^{3})$ and $(u^{\mathrm{in}},E^{\mathrm{in}},B^{\mathrm{in}})\in H^{4}(\mathbb{R}^{3})$ satisfy \eqref{EMI} and assume further $\left\Vert u^{\mathrm{in}}\right\Vert_{\infty}<1$. Let $(\rho,u,B,E)$ be the solution to \eqref{Euler Maxwell} on some time interval $[0,T]$ such that $\left\Vert u\right\Vert_{L^{\infty}([0,T];L^{\infty} (\mathbb{R}^{3}))} <1$ (as guaranteed by Theorem \eqref{local existence of solution}).    
    \item $(\psi^{\mathrm{in}}_{\hbar},A_{0,\hbar}^{\mathrm{in}},A^{\mathrm{in}}_{1,\hbar})\in H^
    {4}(\mathbb{R}^{3})\times H^{4}(\mathbb{R}^{3})\times H^{3}(\mathbb{R}^{3})$ satisfy \eqref{MS initial data} and let  $(\psi_{\hbar},A_{\hbar})$ be the solution of \eqref{Maxwell Schrodinger eq} on $[0,T]$ (as guaranteed by Theorem \ref{Global well posedness of MS}). 
    \item $\mathcal{G}_{\hbar}(t)$ is given by \eqref{corrected modulated energy} and $e_{\hbar}\coloneqq E_{\hbar}-E, b_{\hbar}\coloneqq B_{\hbar}-B$. 
\end{itemize}
Then, it holds that 
\begin{align}
\frac{\dd}{\dd t}\mathcal{G}_{\hbar}(t)&= -\frac{1}{4}\sum_{1\leq j,k\leq 3}\mathrm{tr}\left((i\hbar\partial_{x_{j}}+A^{j}_{\hbar}+u^{j})\vee \left((i\hbar\partial_{x_{k}}+A^{k}_{\hbar}+u^{k})\vee \partial_{x_{k}}u^{j}\right)R_{\hbar}\right) \notag\\
&+\int_{\mathbb{R}^{3}}D_{x}u:\left(e_{\hbar}^{\otimes 2}-\frac{1}{2}\left\vert e_{\hbar}\right\vert^{2}I+b_{\hbar}^{\otimes 2}-\frac{1}{2}\left\vert b_{\hbar}\right\vert^{2}I\right)\ \dd x
-\int_{\mathbb{R}^{3}}\partial_{t}u\cdot (e_{\hbar}\times b_{\hbar})\ \dd x. \label{Time derivative of quantum corrected modulated energy statement} 
\end{align}
Consequently, it holds that 
\begin{align}
\underset{t\in [0,T]}{\sup}\mathcal{H}_{\hbar}(t)\underset{\hbar \rightarrow 0}{\rightarrow} 0\ \mbox{provided}\ \mathcal{H}_{\hbar}(0)\underset{\hbar \rightarrow 0}{\rightarrow}0. \label{vanishing of quantum modulated energy statement}   
\end{align}
Moreover, one has 
\begin{align}
 \underset{t\in [0,T]}{\sup}\left\Vert (E_{\hbar}-E)(t,\cdot)\right\Vert_{2}\underset{\hbar\rightarrow 0}{\rightarrow}0,\ \underset{t\in [0,T]}{\sup}\left\Vert (B_{\hbar}-B)(t,\cdot)\right\Vert_{2}\underset{\hbar\rightarrow 0}{\rightarrow}0 \label{convergence of E and B}
\end{align}
and 
\begin{align}
\underset{t\in [0,T]}{\sup}\left\Vert (\rho_{\hbar}-\rho)(t,\cdot)\right\Vert_{H^{-1}} \underset{\hbar\rightarrow 0}{\rightarrow}0, \ \underset{t\in [0,T]}{\sup}\left\Vert (J_{\hbar}-\rho u)(t,\cdot)\right\Vert_{\mathrm{BL}}\underset{\hbar \rightarrow 0}{\rightarrow}0.  \label{convergence of rho and J}     
\end{align}
\label{main thm}
\end{thm}
The local well-posedness of \eqref{Euler Maxwell} will be addressed in \S\ref{subsection local wellposedness}. As for the system \eqref{Maxwell Schrodinger eq}, we need global well-posedness in order to make sure that the time interval on which the solution $\psi_{\hbar}$ is defined does not degenerate as $\hbar\rightarrow 0$.  Below we state the global in time well-posedness of \eqref{Maxwell Schrodinger eq}, as proved in \cite{NakamuraWada2007}.
\begin{thm} \textup{(\cite{NakamuraWada2007}, Theorem 1.1 and Theorem 1.2)} \label{Global well posedness of MS}
Suppose that $(\psi^{\mathrm{in}}_{\hbar},A^{\mathrm{in}}_{0,\hbar},A^{\mathrm{in}}_{1,\hbar})\in H^{4}(\mathbb{R}^{3})\times H^{4}(\mathbb{R}^{3})\times H^{3}(\mathbb{R}^{3})$ satisfies \eqref{MS initial data}. Then there exists a unique solution $(\psi_{\hbar},A_{\hbar})\in C(\mathbb{R};H^{4}(\mathbb{R}^{3}))\times C(\mathbb{R};H^{4}(\mathbb{R}^{3}))$ to the Cauchy problem \eqref{Maxwell Schrodinger eq}-\eqref{MS initial data}. Moreover, $\partial_{t}A_{\hbar}\in C(\mathbb{R};H^{3}(\mathbb{R}^{3}))$. 
\end{thm}
For additional literature concerning the global well-posedness for the equation \eqref{Maxwell Schrodinger eq} (or variants thereof) we refer the reader to \cite{Tsutsumi1993, BejenaruTataru2009, AntonelliMarcatiScandone2022}. 
\subsection{Strategy and key ideas.} Prior to establishing the semiclassical limit leading from a dispersive to a monokinetic PDE, one has to first gain a proper understanding of the \textit{weak-strong stability} exhibited by the underlying monokinetic PDE. To be  concrete,  given solutions $(\rho_{1},u_{1},E_{1},B_{1})$ and $(\rho_{2},u_{2},E_{2},B_{2})$ to \eqref{Euler Maxwell} define the \textit{modulated energy} by 
\begin{align}
\mathcal{H}(t)&\coloneqq \underset{\coloneqq \mathcal{U}(t)}{\underbrace{\frac{1}{2}\int_{\mathbb{R}^{3}}\rho_{1}\left\vert u_{1}-u_{2}\right\vert^{2}\ \dd x}}+\underset{\coloneqq \mathcal{E}(t)}{\underbrace{\frac{1}{2}\int_{\mathbb{R}^{3}}\left\vert E_{1}-E_{2}\right\vert^{2}\ \dd x}}+\underset{\coloneqq \mathcal{B}(t)}{\underbrace{\frac{1}{2}\int_{\mathbb{R}^{3}}\left\vert B_{1}-B_{2}\right\vert^
{2}\ \dd x}} \label{modulated energy def} 
\end{align}
and the \textit{corrected} modulated energy by 
\begin{align}
\mathcal{G}(t)\coloneqq\mathcal{H}(t)-\int_{\mathbb{R}^{3}}u_{2}\cdot((E_{1}-E_{2})\times (B_{1}-B_{2}))\ \dd x. \label{def of corrected modulated energy}   
\end{align}
Note that $\mathcal{H}(t)$ and $\mathcal{G}(t)$ are the classical analogues of \eqref{def of quantum modualted energy} and \eqref{corrected modulated energy}, respectively. A weak-strong stability principle would typically be an evolution estimate of the type 
\begin{align*}
\frac{\dd}{\dd t}(\mbox{modulated energy})\leq C(\mbox{modulated energy})\ \mbox{with}\ C=C(\left\Vert \rho_{2}\right\Vert_{\infty},\left\Vert u_{2}\right\Vert_{W^{1,\infty}_{t,x}([0,T]\times \mathbb{R}^{3})}). \end{align*}
Crucially, the constant $C$ involves only the regularity of one of the solutions.   A weak-strong stability inequality is well established for the pressureless Euler-Poisson equation -- see appendix A. in \cite{Serfaty2020MeanField}. An attempt to mimic the proof of this weak-strong stability inequality would  include the calculation of the time derivative of the modulated energy \eqref{modulated energy def}. This calculation is performed in Theorem \ref{thm about der of modualted energy} and results in the following formula 
\begin{align}
\frac{\dd}{\dd t}\mathcal{H}(t)=&-\int_{\mathbb{R}^{3}}\rho_{1}(u_{1}-u_{2})^{\otimes 2}:D_{x}u_{2}\ \dd x -\int_{\mathbb{R}^{3}}(\rho_{1}-\rho_{2})u_{2}\cdot (E_{1}-E_{2})\ \dd x \notag\\
&+\int_{\mathbb{R}^{3}}\rho_{1}(u_{1}-u_{2})\cdot (u_{2}\times (B_{1}-B_{2}))\ \dd x.  \label{formula for time der of H intro}   
\end{align}
It is transparent that the first term in the right-hand side of \eqref{formula for time der of H intro} can be bounded by means of $\mathcal{U}(t)$  up to a constant depending only on the regularity of $u_{2}$ -- see \S\ref{weak strong subsection} for details. However, the last term in \eqref{formula for time der of H intro} would require an $L^{\infty}$ bound on $\rho_{1}$, and therefore such a bound would not enable to obtain a weak-strong stability principle. We stress that we wish to avoid at all costs any dependence on the $L^{\infty}$ norm of $\rho_{1}$: in the semiclassical limit $\rho_{1}$ will  play the role of $\rho_{\hbar}$, and we already remarked that it is currently not possible to rule out that $\underset{t\in [0,T]}{\sup}\left\Vert \rho_{\hbar}(t,\cdot)\right\Vert_{\infty} $ blows up as $\hbar\rightarrow 0$. The idea of the corrected modulated energy \eqref{def of corrected modulated energy} is to circumvent this issue. Indeed, once the correction is incorporated within the modulated energy it is possible to deduce a weak-strong inequality for small initial data by observing that $\mathcal{G}$ is bounded from below and above in terms of $\mathcal{H}$ as long as $u_{2}$ remains small (see Lemma \ref{sandwich lemma}). More precisely it holds that 
\begin{align*}
(1-\Lambda)\mathcal{H}(t)\leq \mathcal{G}(t)\leq (1+\Lambda)\mathcal{H}(t)\ \mbox{as long as}\ \left\Vert u_{2}\right\Vert_{\infty}\leq \Lambda \ \mbox{for some}\ \Lambda<1.     
\end{align*}
The plan outlined above enables to prove a weak-strong stability principle for \eqref{Euler Maxwell} under smallness assumptions. Conceptually, this is similar in spirit to the adapted modulated energy implemented in \cite{salvi2025semi}, and a major part of the novelty of the present work is that this idea can be imported to the settings of the Maxwell-Schr\"odinger equation. This component is also the reason the final result is unconditional, in variance with \cite{KimMoon2024}. Having bypassed this issue, the next task is to quantize the weak-strong stability argument, which is where the quantum modulated energies \eqref{def of quantum modualted energy}-\eqref{corrected modulated energy} come into play.

\vspace{0.4 cm}

The paper is organized as follows. In \S\ref{wellposednes section} we prove the local well-posedness of \eqref{Euler Maxwell}. Note that since we are considering the pressureless Euler-Maxwell equation local well-posedness is natural (in variance with \cite{KimMoon2024}, which includes a pressure term leading to a global theory). Afterwards, we prove the weak-strong stability principle by studying the evolution of the corrected modulated energy $\mathcal{G}(t)$. In \S\ref{main thm section} we present a proof of the semiclassical limit leading from \eqref{Maxwell Schrodinger eq} to \eqref{Euler Maxwell} as formulated in Theorem \ref{main thm}.     
\section{Local well-posedness and weak-strong stability for \eqref{Euler Maxwell}}\label{wellposednes section}
\subsection{Local in time existence of classical solutions to \eqref{Euler Maxwell}}\label{subsection local wellposedness} We start by demonstrating local in time existence of classical solutions to \eqref{Euler Maxwell}. The proof differs substantially from the local well-posedness theory of Euler-Poisson due to the fact that $E$ and $B$ are governed dynamically according to a hyperbolic equation. The local existence is proved only for small initial data, which is more restrictive than the local existence for the Euler-Poisson equation. The following theorem goes back to the work of Kato \cite{kato1975cauchy} and has been revised in \cite{taylor1991paradifferential}.  
\begin{thm}\textup{(\cite{taylor1991paradifferential}, Chapter 5)}\label{Kato type thm}
Suppose that: 
\begin{itemize}
    \item[i.] $\Omega\subset \mathbb{R}^{n}$ is open with $0\in \Omega$ and $\mathbf{A}^{0},\dots \mathbf{A}^{d}\in C^{\infty}(\Omega;M_{n\times n}(\mathbb{R}))$, $F\in C^{\infty}(\Omega;\mathbb{R}^{n})$ and $F(0)=0$. 
    \item[ii.] For all $j=0,\dots,d $ it holds that 
    \begin{align}\mathbf{A}^{j}(\zeta)=(\mathbf{A}^{j})^{T}(\zeta) \ \mbox{for all}\ \zeta\in \Omega  \label{symmetry of Aj}
    \end{align}
    and 
    \begin{align}
      \xi^{T}\mathbf{A}^{0}(\zeta)\xi>0 \ \mbox{for all}\ \xi\in \mathbb{R}^{n}\setminus\{0\} \ \mbox{and} \ \zeta\in \Omega.  \label{positivity of A0}
    \end{align}
    \item[iii.] $W^{\mathrm{in}}\in H^{s}(\mathbb{R}^{d};\mathbb{R}^{n})$ for some $s>\frac{d}{2}+1$ and $\overline{\left\{W^{\mathrm{in}}(x)\vert x\in \mathbb{R}^{d}\right\}}\Subset \Omega$.
\end{itemize}
Then, there exist some $T>0$ and a unique solution $W\in C([0,T];H^{s}(\mathbb{R}^{d}))\cap C^{1}([0,T];H^{s-1}(\mathbb{R}^{d}))$ to the Cauchy problem
\begin{align*}
\mathbf{A}^{0}(W)\partial_{t}W+\sum_{j=1}^{d}\mathbf{A}^{j}(W)\partial_{x_{j}}W=F(W), \ W(0,\cdot)=W^{\mathrm{in}}.     
\end{align*}
Moreover, $W(t,x)\in \Omega$ for all $(t,x)\in [0,T]\times \mathbb{R}^{d}$. 
\end{thm}
\begin{thm}\label{local existence of solution}
Suppose that $\rho^{\mathrm{in}}\in H^{s-1}(\mathbb{R}^{3})$ and
$(u^{\mathrm{in}},E^{\mathrm{in}},B^
{\mathrm{in}})\in H^{s}(\mathbb{R}^{3})$ satisfy \eqref{EMI} for some $s>\frac{5}{2}$ and suppose that $\left\Vert u^{\mathrm{in}}\right\Vert_{\infty}<1 $.  Then, there exist some $T>0$ and a solution $(\rho,u,E,B)$  to \eqref{Euler Maxwell} with initial data 
$(\rho^{\mathrm{in}},u^{\mathrm{in }},E^{\mathrm{in}},B^{\mathrm{in}})$. Moreover, it holds that $\rho\in C([0,T];H^{s-1}(\mathbb{R}^{3}))$ and  $(u,E,B)\in C([0,T];H^{s}(\mathbb{R}^{3}))$ with the bound $\left\Vert u\right\Vert_{L^{\infty}([0,T];L^{\infty}(\mathbb{R}^{3}))} <1$. 
\end{thm}
\begin{proof}
\textbf{Step 1. Reduction of the system}. We claim that existence of a solution to \eqref{Euler Maxwell}-\eqref{EMI} would follow from existence of a solution to the system 
\begin{align}
\begin{cases}
 \begin{array}{lc}
        \partial_{t}u+uD_{x}u=E+u\times B\\
       \partial_{t}E-\mathrm{curl}B+u\mathrm{div}E=0\\
       \partial_{t}B+\mathrm{curl}E+u\mathrm{div}B=0\\
       (u(0,\cdot),E(0,\cdot),B(0,\cdot))=(u^{\mathrm{in}},E^{\mathrm{in}},B^{\mathrm{in}}).  
 \end{array}
 \end{cases}
\label{reduced Euler Maxwell}
\end{align}
To see this, suppose $(u,E,B)\in C([0,T];H^{s}(\mathbb{R}^{3}))$ is a solution to \eqref{reduced Euler Maxwell} and set $\rho=\mathrm{div}E$ (so that $\rho(0,\cdot)=\mathrm{div}E^{\mathrm{in}}=\rho^{\mathrm{in}}$). Then,
by taking the divergence in the equation for $B$ and using that $\mathrm{div}(\mathrm{curl})\equiv 0$ we get 
\begin{align*}
\partial_{t}\mathrm{div}B+\mathrm{div}(u\mathrm{div}B)=0,\ \mathrm{div}B(0,\cdot )=0.       
\end{align*}
Thus $\mathrm{div}B$ satisfies a linear transport equation in divergence form with $0$ initial data, and therefore by uniqueness it must be that $\mathrm{div}B\equiv 0$. 
Moreover, setting $\rho=\mathrm{div}E$ the equation for $E$ writes 
\begin{align*}
\partial_{t}E-\mathrm{curl}B=-\rho u.    
\end{align*}
Taking the divergence of the last equation we get 
\begin{align*}
\partial_{t}\rho+\mathrm{div}(\rho u)=0,  \end{align*}
so that we eventually arrive at 
\begin{align}
\begin{cases}
\begin{array}{lc}
\partial_{t}\rho+\mathrm{div}(\rho u)=0\\
\partial_{t}u+uD_{x}u=E+u\times B\\
\partial_{t}E-\mathrm{curl}B=-\rho u\\
\partial_{t}B+\mathrm{curl}E=0\\
\mathrm{div}E=\rho,\ \mathrm{div}(B)=0,
\end{array}
\end{cases}
\end{align}
which is \eqref{Euler Maxwell}. Note also that since $u\in C([0,T];H^{s}(\mathbb{R}^{3}))$ we also have $\rho\in C([0,T];H^{s-1}(\mathbb{R}^{3}))$. 
\textbf{Step 2. The system \eqref{reduced Euler Maxwell} in vector form.} Set $W\coloneqq(u,E,B)$ and $$F(u,E,B)\coloneqq \left(\begin{array}{c}
E+u\times B\\
0\\
0
\end{array}\right). $$
Given a vector $(a^{1},a^{2},a^{3})$ denote 
\begin{align*}
\mathbf{C}(a)\coloneqq \begin{pmatrix}0 & -a^{3} & a^{2}\\
a^{3} & 0 & -a^{1}\\
-a^{2} & a^{1} & 0
\end{pmatrix}    
\end{align*}
and denote by $\mathbf{E}_{j}(a)$ the $3\times 3$ matrix whose $j$-th column is $a$ and $0$ at all other entries. 
Note that $a\times b=\mathbf{C}(a)b$. Finally, 
we designate by  $e_{j}$ the $j$-th unit vector and by $\mathbf{B}^{j}:\mathbb{R}^{9}\rightarrow M_{9\times 9}(\mathbb{R})$ ($j=1,2,3$) the matrix given by 
\begin{align*}
\mathbf{B}^{j}(u,E,B)\coloneqq \begin{pmatrix}u^{j}I & 0 & 0\\
0 & \mathbf{E}_{j}(u) & -\mathbf{C}(e_{j})\\
0 & \mathbf{C}(e_{j}) & \mathbf{E}_{j}(u)
\end{pmatrix}.    \end{align*}
With this notation the system \eqref{reduced Euler Maxwell} writes 
\begin{align}
\partial_{t}W+\sum_{j=1}^{3}\mathbf{B}^{j}\partial_{x_{j}}W=F(W), \ W(0,\cdot)=W^{\mathrm{in}}. \label{vector form system}      
\end{align}
\textbf{Step 3. Symmetrization.} Note that the matrices $\mathbf{B}^{j}$ in \eqref{vector form system} are not symmetric as required by Theorem \ref{Kato type thm}. Therefore we transform the system to match the conditions in Theorem \ref{Kato type thm} as follows. Set 
\begin{align*}
\mathbf{S}(u,E,B)\coloneqq \begin{pmatrix}I & 0 & 0\\
0 & I & \mathbf{C}(u)\\
0 & -\mathbf{C}(u) & I
\end{pmatrix}\in M_{9\times 9}(\mathbb{R}).    
\end{align*}
We start by verifying that $\mathbf{S}\mathbf{B}^{j}$ is symmetric. 
Invoking block-matrix multiplication we compute that: 
\begin{align}
\mathbf{S}(u,E,B)\mathbf{B}^{j}(u,E,B)=\begin{pmatrix}u^{j}I & 0 & 0\\
0 & \mathbf{E}_{j}(u)+\mathbf{C}(u)\mathbf{C}(e_{j}) & -\mathbf{C}(e_{j})+\mathbf{C}(u)\mathbf{E}_{j}(u)\\
0 & \mathbf{C}(e_{j})-\mathbf{C}(u)\mathbf{E}_{j}(u) & \mathbf{E}_{j}(u)+\mathbf{C}(u)\mathbf{C}(e_{j})
\end{pmatrix}. 
\label{Multiplication of S by Aj}
\end{align}
By direct calculation we have 
\begin{align*}
\mathbf{E}_{j}(u)+\mathbf{C}(u)\mathbf{C}(e_{j})=\mathbf{E}_{j}(u)+\mathbf{E}_{j}^{T}(u)-u_{j}I    
\end{align*}
which is readily seen to be symmetric. 
Therefore all the matrices on the diagonal in \eqref{Multiplication of S by Aj} are symmetric. In addition we have 
\begin{align*}
\mathbf{C}(a)^{T}=\begin{pmatrix}0 & a^{3} & -a^{2}\\
-a^{3} & 0 & a^{1}\\
a^{2} & -a^{1} & 0
\end{pmatrix}=-\mathbf{C}(a)   
\end{align*}
and
\begin{align*}
\mathbf{C}(u)\mathbf{E}_{j}(u)=\mathbf{E}_{j}(u\times u)=\mathbf{0}_{3\times 3}.   
\end{align*}
This shows that $\mathbf{S}\mathbf{B}^{j}$ is symmetric, i.e. that condition \eqref{symmetry of Aj} is verified. 
Next we verify the condition \eqref{positivity of A0}. Given $\xi=(v,E,B)\in \mathbb{R}^{9}\setminus\{0\}$ and $\zeta=(u,E',B')$ we have 
\begin{align*}
\xi^{T}\mathbf{S}(u,E',B')\xi&=(v,E,B)\begin{pmatrix}v \\
E+\mathbf{C}(u)B \\
-\mathbf{C}(u)E+B 
\end{pmatrix}=\left\vert v\right\vert^{2}+\left\vert E\right\vert^{2}+E\cdot (u\times B)+\left\vert B\right\vert^{2}-B\cdot (u\times E)\\
&=\left\vert v\right\vert^{2}+\left\vert E\right\vert^{2}+\left\vert B\right\vert^{2}-2u\cdot (E\times B). 
\end{align*}
Note that for all $(v,E,B)\neq 0$ and $\left\vert u\right\vert <1$ it holds that  
\begin{align*}
\left\vert v\right\vert^{2}+\left\vert E\right\vert^{2}+\left\vert B\right\vert^{2}-2u\cdot (E\times B)&\geq \left\vert v\right\vert^{2}+\left\vert E\right\vert^{2}+\left\vert B\right\vert^{2}-\left\vert u\right\vert(\left\vert B\right\vert^{2}+\left\vert E\right\vert^{2})\\
&=\left\vert v\right\vert^{2}+(1-\left\vert u\right\vert)(\left\vert E\right\vert^{2}+\left\vert B\right\vert^{2})>0.         
\end{align*}
Fix $\delta >0$ such that $\left\Vert u^{\mathrm{in}}\right\Vert_{\infty}\leq  1-\delta$. Then $\mathbf{S}$ verifies \eqref{positivity of A0} on $$\Omega\coloneqq \{(u,E,B)\in \mathbb{R}^{9}\vert \left\vert u\right\vert< 1-\frac{\delta}{2}\}.$$ 
Moreover, by assumption we also have $\overline{\{(u^{\mathrm{in}}(x),E^{\mathrm{in}}(x),B^{\mathrm{in}}(x))\vert x\in \mathbb{R}^{3}\}}\Subset \Omega$, which shows that condition iii. is satisfied. To conclude, by Theorem \ref{Kato type thm} we have proved that there exist some $T>0$ and a unique solution $W$ to the system 
\begin{align*}
\mathbf{S}(W)\partial_{t}W+\sum_{j=1}^{3}\mathbf{S}(W)\mathbf{B}^{j}\partial_{x_{j}}W=\mathbf{S}(W)F(W), \ W(0,\cdot)=W^{\mathrm{in}}.    
\end{align*}
Since $\mathbf{S}(W)$ is invertible on $[0,T]$ by applying the inverse we get that $W$ is the requested solution for the system \eqref{vector form system}. Finally we have that $W(t,x)\in \Omega$ for all $(t,x)\in [0,T]\times \mathbb{R}^{3}$ so that $\left\Vert u\right\Vert_{\infty}<1 $. 
\end{proof}
\subsection{Weak-strong stability for \eqref{Euler Maxwell}}\label{weak strong subsection} In this section we prove a weak-strong stability inequality for the Euler-Maxwell equation \eqref{Euler Maxwell}. As mentioned in the introduction, the weak-strong stability principle already captures one key idea which reflects the novelty of this work, namely incorporating a corrector within the modulated energy. The proof of the semiclassical limit in the next section will be modeled on this inequality. The main theorem of this section is stated below.  
\begin{thm} Suppose that: 
\begin{itemize}
    \item $(\rho_{1},u_{1},E_{1},B_{1}), (\rho_{2},u_{2},E_{2},B_{2})$ are solutions to \eqref{Euler Maxwell} on $[0,T]$ with initial data $$(\rho_{1}^{\mathrm{in}},u^{\mathrm{in}}_{1},E^{\mathrm{in}}_{1},B^{\mathrm{in}}_{1}),(\rho_{2}^{\mathrm{in}},u^{\mathrm{in}}_{2},E^{\mathrm{in}}_{2},B^{\mathrm{in}}_{2})$$ respectively.
    \item $\mathcal{G}(t)$ is given by \eqref{def of corrected modulated energy} and $e\coloneqq E_{1}-E_{2},b\coloneqq B_{1}-B_{2}$.  
\end{itemize} 
Then, it holds that 
\begin{align}
\frac{\dd}{\dd t}\mathcal{G}(t)
&=-\int_{\mathbb{R}^{3}}\rho_{1}(u_{1}-u_{2})^{\otimes 2}:D_{x}u_{2}\ \dd x
\notag\\
&+\int_{\mathbb{R}^{3}}D_{x}u_{2}:
\left(
e^{\otimes 2}-\frac{1}{2}|e|^{2}I
+b^{\otimes 2}-\frac{1}{2}|b|^{2}I
\right)\ \dd x
-\int_{\mathbb{R}^{3}}\partial_{t}u_{2}\cdot(e\times b)\ \dd x.
\label{eq:time-der-G}
\end{align}
Consequently, if $\left\Vert u_{2}\right\Vert_{L^\infty([0,T];L^{\infty}( \mathbb{R}^{3}))}< 1$  then it holds that 
\begin{align}
\mathcal{H}(t)\leq Ce^{Ct}\mathcal{H}(0)\ \mbox{for all}\ t\in [0,T] 
\label{H bound statement}   
\end{align}
for some constant $C=C(\left\Vert \rho_{2}\right\Vert_{\infty} ,\left\Vert u_{2}\right\Vert _{W^{1,\infty}_{t,x}([0,T]\times \mathbb{R}^{3})})>0$. 
\label{time der of corrected modulated energy thm} 
\end{thm}
We start by calculating the time variation of $\mathcal{H}(t)$, which is the content of the following theorem. We shall invoke the abbreviations $\mathcal{U}(t),\mathcal{E}(t),\mathcal{B}(t)$ as defined in \eqref{modulated energy def}. 
\begin{thm}\label{thm about der of modualted energy}
Suppose that: 
\begin{itemize}
    \item $(\rho_{1},u_{1},E_{1},B_{1}), (\rho_{2},u_{2},E_{2},B_{2})$ are solutions to \eqref{Euler Maxwell} with initial data $$(\rho_{1}^{\mathrm{in}},u^{\mathrm{in}}_{1},E^{\mathrm{in}}_{1},B^{\mathrm{in}}_{1}),(\rho_{2}^{\mathrm{in}},u^{\mathrm{in}}_{2},E^{\mathrm{in}}_{2},B^{\mathrm{in}}_{2})$$ respectively.
    \item $\mathcal{H}(t)$ is given by \eqref{modulated energy def}. 
\end{itemize}
Then, it holds that 
\begin{align}
\frac{\dd}{\dd t}\mathcal{H}(t)=&-\int_{\mathbb{R}^{3}}\rho_{1}(u_{1}-u_{2})^{\otimes 2}:D_{x}u_{2}\ \dd x -\int_{\mathbb{R}^{3}}(\rho_{1}-\rho_{2})u_{2}\cdot (E_{1}-E_{2})\ \dd x \notag\\
&+\int_{\mathbb{R}^{3}}\rho_{1}(u_{1}-u_{2})\cdot (u_{2}\times (B_{1}-B_{2}))\ \dd x. \label{time der of modulated energy}      
\end{align}   
\end{thm}
\begin{proof}
\textbf{Time variation of $\mathcal{U}(t)$.} We claim that the following formula holds  \begin{align}
\frac{\dd}{\dd t}\mathcal{U}(t)=&\int_{\mathbb{R}^{3}}\rho_{1}(u_{1}-u_{2})\cdot (E_{1}-E_{2})\ \dd x
+\int_{\mathbb{R}^{3}}\rho_{1}(u_{1}-u_{2})\cdot (u_{2}\times (B_{1}-B_{2}))\ \dd x \notag\\
&-\int_{\mathbb{R}^{3}}\rho_{1}(u_{1}-u_{2})^{\otimes 2}:D_{x}u_{2}\ \dd x. \label{Time der of kinetic part formula}   \end{align}
For brevity we set $w\coloneqq u_{1}-u_{2}$ so that  
\begin{align*}
\frac{\dd}{\dd t}\mathcal{U}(t)=\frac{1}{2}\int_{\mathbb{R}^{3}}\partial_{t}\rho_{1}\left\vert w\right\vert^{2}\ \dd x+\int_{\mathbb{R}^{3}}\rho_{1}w\cdot \partial_{t}w\ \dd x.   \end{align*}
Note that $w$ is governed by the following equation: 
\begin{align*}
\partial_{t}w+u_{1}D_{x}u_{1}-u_{2}D_{x}u_{2}=E_{1}+u_{1}\times B_{1}-E_{2}-u_{2}\times B_{2},     
\end{align*}
which is recast as 
\begin{align*}
\partial_{t}w+u
_{1}D_{x}w+wD_{x}u_{2}=E_{1}-E_{2}-u_{2}\times (B_{2}-B_{1})+w\times B_{1}. 
\end{align*}
Therefore, we obtain  
\begin{align*}
\frac{\dd}{\dd t}\mathcal{U}(t)=&\int_{\mathbb{R}^{3}}\rho_{1}u_{1}\cdot (D_{x}w w^{T})\ \dd x\\
&+\int_{\mathbb{R}^{3}}\rho_{1}w\cdot (E_{1}-E_{2}-u_{2}\times (B_{2}-B_{1})+w\times B_{1})\ \dd x\\
&-\int_{\mathbb{R}^{3}}\rho_{1}w\cdot(u_{1}D_{x}w+wD_{x}u_{2})\ \dd x\\
=& \int_{\mathbb{R}^{3}}\rho_{1}w\cdot (E_{1}-E_{2})\ \dd x
-\int_{\mathbb{R}^{3}}\rho_{1}w^{\otimes 2}:D_{x}u_{2}\ \dd x\\
&+\int_{\mathbb{R}^{3}}\rho_{1}u_{1}\cdot (D_{x}w w^{T})\ \dd x-\int_{\mathbb{R}^{3}}\rho_{1}w\cdot (u_{2}\times (B_{2}-B_{1}))\ \dd x-\int_{\mathbb{R}^{3}}\rho_{1}w\cdot (u_{1}D_{x}w)\ \dd x. 
\end{align*}
Note that 
\begin{align*}
w\cdot (u_{1}D_{x}w)=u_{1}\cdot (D_{x}w w^{T})  
\end{align*}
and therefore 
\begin{align*}
\int_{\mathbb{R}^{3}}\rho_{1}u_{1}\cdot (D_{x}w w^{T})\ \dd x-\int_{\mathbb{R}^{3}}\rho_{1}w\cdot (u_{1}D_{x}w)\ \dd x=0.     
\end{align*}
Consequently we conclude that 
\begin{align*}
\frac{\dd}{\dd t}\mathcal{U}(t)=&\int_{\mathbb{R}^{3}}\rho_{1}w\cdot (E_{1}-E_{2})\ \dd x+\int_{\mathbb{R}^{3}}\rho_{1}w\cdot (u_{2}\times (B_{1}-B_{2}))\ \dd x\\
&-\int_{\mathbb{R}^{3}}\rho_{1}w^{\otimes 2}:D_{x}u_{2}\ \dd x, 
\end{align*}
which establishes \eqref{Time der of kinetic part formula}.\\ 
\textbf{Calculation of $\frac{\dd}{\dd t}\mathcal{E}(t)+\frac{\dd}{\dd t}B(t)$.} We claim there holds the formula 
\begin{align}
\frac{\dd}{\dd t}\mathcal{E}(t)+\frac{\dd }{\dd t}\mathcal{B}(t)=\int_{\mathbb{R}^{3}}(E_{1}-E_{2})\cdot (\rho_{2}u_{2}-\rho_{1}u_{1})\ \dd x. \label{time der of maxwell part}    
\end{align}
We compute that 
\begin{align*}
\frac{\dd}{\dd t} \mathcal{E}(t)&=\int_{\mathbb{R}^{3}}(E_{1}-E_{2})\cdot(\mathrm{curl}(B_{1})-\rho_{1}u_{1}-\mathrm{curl}(B_{2})+\rho_{2}u_{2})\ \dd x \notag\\
&=\int_{\mathbb{R}^{3}}(E_{1}-E_{2})\cdot \mathrm{curl}(B_{1}-B_{2})\ \dd x+\int_{\mathbb{R}^{3}}(E_{1}-E_{2})\cdot (\rho_{2}u_{2}-\rho_{1}u_{1})\ \dd x 
\end{align*}
and 
\begin{align*}
\frac{\dd}{\dd t}\mathcal{B}(t)=\int_{\mathbb{R}^{3}}(B_{1}-B_{2})\cdot \mathrm{curl}(E_{2}-E_{1})\ \dd x.   
\end{align*}
Integration by parts yields 
\begin{align*}
\int_{\mathbb{R}^{3}}(E_{1}-E_{2})\cdot \mathrm{curl}(B_{1}-B_{2})\ \dd x+\int_{\mathbb{R}^{3}}(B_{1}-B_{2})\cdot \mathrm{curl}(E_{2}-E_{1})\ \dd x=0.      
\end{align*}
Consequently, we deduce that 
\begin{align*}
\frac{\dd}{\dd t}\mathcal{E}(t)+\frac{\dd}{\dd t}\mathcal{B}(t)=\int_{\mathbb{R}^{3}}(E_{1}-E_{2})\cdot (\rho_{2}u_{2}-\rho_{1}u_{1})\ \dd x 
\end{align*}
which establishes \eqref{time der of maxwell part}.\\
\textbf{Conclusion.} From \eqref{Time der of kinetic part formula} and \eqref{time der of maxwell part} we conclude that 
\begin{align*}
\frac{\dd}{\dd t}\mathcal{H}(t)=&\int_{\mathbb{R}^{3}}\rho_{1}(u_{1}-u_{2})\cdot (E_{1}-E_{2})\ \dd x+\int_{\mathbb{R}^{3}}(E_{1}-E_{2})\cdot (\rho_{2} u_{2}-\rho_{1}u_{1})\ \dd x\\
&+\int_{\mathbb{R}^{3}}\rho_{1}(u_{1}-u_{2})\cdot (u_{2}\times (B_{1}-B_{2}))\ \dd x
-\int_{\mathbb{R}^{3}}\rho_{1}(u_{1}-u_{2})^{\otimes 2}:D_{x}u_{2}\ \dd x\\
=&-\int_{\mathbb{R}^{3}}\rho_{1}(u_{1}-u_{2})^{\otimes 2}:D_{x}u_{2}\ \dd x -\int_{\mathbb{R}^{3}}(\rho_{1}-\rho_{2})u_{2}\cdot (E_{1}-E_{2})\ \dd x\\
&+\int_{\mathbb{R}^{3}}\rho_{1}(u_{1}-u_{2})\cdot (u_{2}\times (B_{1}-B_{2}))\ \dd x,
\end{align*} 
as desired. 
\end{proof}
The following simple lemma shows that for $u_{2}$ small enough the corrected modulated energy $\mathcal{G}(t)$ is bounded from below and above by means of $\mathcal{H}(t)$. This is precisely where the smallness assumption on  $u_{2}$ comes into play.  
\begin{lem}
Let the assumptions of Theorem \ref{thm about der of modualted energy} hold and assume in addition that $\left\Vert u_{2}\right\Vert_{\infty} \leq \Lambda$ for some $\Lambda <1 $. Then, it holds that   
\begin{align*}
(1-\Lambda)\mathcal{H}(t)\leq \mathcal{G}(t)\leq (1+\Lambda) \mathcal{H}(t).      
\end{align*}
\label{sandwich lemma}
\end{lem}
\begin{proof}
We estimate that 
\begin{align*}
&\left\vert \int_{\mathbb{R}^{3}}u_{2}\cdot ((E_{1}-E_{2})\times (B_{1}-B_{2}))\ \dd x\right\vert\leq \left\Vert u_{2}\right\Vert_{\infty}\int_{\mathbb{R}^{3}}\left\vert E_{1}-E_{2}\right\vert\left\vert B_{1}-B_{2}\right\vert\ \dd x\\
&\leq \Lambda\Big(\frac{1}{2}\int_{\mathbb{R}^{3}}\left\vert E_{1}-E_{2}\right\vert^{2}\ \dd x+\frac{1}{2}\int_{\mathbb{R}^
{3}}\left\vert B_{1}-B_{2}\right\vert^{2}\ \dd x  \Big)\leq \Lambda (\mathcal{E}(t)+\mathcal{B}(t)).    \end{align*}
Therefore we conclude that 
\begin{align*}
(1-\Lambda)\mathcal{H}(t)\leq \mathcal{G}(t)\leq  (1+\Lambda) \mathcal{H}(t),   
\end{align*}
as asserted. 
\end{proof}
\begin{rem}
The same inequality holds for the quantum modulated energies $\mathcal{H}_{\hbar}, \mathcal{G}_{\hbar}$. Indeed, by exactly the same considerations we have 
\begin{align*}
(1-\Lambda)\mathcal{H}_{\hbar}(t)\leq \mathcal{G}_{\hbar}(t)\leq (1+\Lambda)\mathcal{H}_{\hbar}(t)    
\end{align*}
provided that $\left\Vert u\right\Vert_{L^{\infty}([0,T];L^{\infty}(\mathbb{R}^{3}))}\leq \Lambda$ for some $0<\Lambda <1$. 
\end{rem}
We are now well positioned to conclude the proof of Theorem \ref{time der of corrected modulated energy thm}.

\vspace{0.4 cm}

\textbf{Proof of Theorem \eqref{time der of corrected modulated energy thm}.}
\textbf{The identity \eqref{eq:time-der-G}}. Recall that we have set $w=u_{1}-u_{2}, e=E_{1}-E_{2}$ and $b=B_{1}-B_{2}$. By Theorem \ref{thm about der of modualted energy} we have 
\begin{align}
\frac{\dd}{\dd t}\mathcal{G}(t)&=-\int_{\mathbb{R}^{3}}\rho_{1}w^{\otimes 2}:D_{x}u_{2}\ \dd x -\int_{\mathbb{R}^{3}}(\rho_{1}-\rho_{2})u_{2}\cdot e\ \dd x \notag\\
&+\int_{\mathbb{R}^{3}}\rho_{1}w\cdot (u_{2}\times b)\ \dd x-\int_{\mathbb{R}^{3}}\partial_{t}u_{2}\cdot (e\times b)\ \dd x \notag\\
&-\int_{\mathbb{R}^{3}}u_{2}\cdot (\partial_{t}e\times b+e\times \partial_{t}b)\ \dd x\coloneqq \sum_{j=1}^{5}\mathcal{T}_{j}. \label{formula for time der of G}   
\end{align}
One has the following identities:  
\begin{align*}
\partial_{t}e&=-\rho_{1}u_{1}+\mathrm{curl}(B_{1})+\rho_{2}u_{2}-\mathrm{curl}(B_{2})=\rho_{2}u_{2}-\rho_{1}u_{1}+\mathrm{curl}(B_{1}-B_{2})\\
&= \rho_{2}u_{2}-\rho_{1}u_{1}+\mathrm{curl}(b) 
\end{align*}
and 
\begin{align*}
\partial_{t}b=\mathrm{curl}(E_{2})-\mathrm{curl}(E_{1})=-\mathrm{curl}(e).     
\end{align*}
Therefore, it follows that  
\begin{align}
\mathcal{T}_{5}=
\int_{\mathbb{R}^{3}} u_{2}\cdot (e\times \mathrm{curl}(e))\ \dd x
-\int_{\mathbb{R}^{3}} u_{2}\cdot (\mathrm{curl}(b)\times b)\ \dd x
-\int_{\mathbb{R}^{3}} u_{2}\cdot ((\rho_{2}u_{2}-\rho_{1}u_{1})\times b)\ \dd x. \label{wedge product term}
\end{align}
Recall that for a vector field $v:\mathbb{R}^{3}\rightarrow \mathbb{R}^{3}$ there holds the identity 
\begin{align*}
\mathrm{curl}(v)\times v=\mathrm{div}\left(v
^{\otimes 2} -\frac{1}{2}\left\vert v\right\vert^{2} I\right)-(\mathrm{div}v)v.   
\end{align*}
Therefore, we find that 
\begin{align}
\int_{\mathbb{R}^{3}} u_{2}\cdot (e\times \mathrm{curl}(e))\ \dd x&=-\int_{\mathbb{R}^{3}} u_{2}\cdot \mathrm{div}\left(e^{\otimes 2}-\frac{1}{2}\left\vert e\right\vert^{2}I \right)\ \dd x+\int_{\mathbb{R}^{3}}(\mathrm{div}e)e\cdot u_{2}\ \dd x \notag\\
&=\int_{\mathbb{R}^{3}} D_{x}u_{2}:\left(e^{\otimes 2} -\frac{1}{2}\left\vert e\right\vert^{2}I \right)\ \dd x+\int_{\mathbb{R}^{3}}(\mathrm{div}e)e\cdot u_{2}\ \dd x  
\label{e times curl e}
\end{align}
and that 
\begin{align}
-\int_{\mathbb{R}^{3}} u_{2}\cdot (\mathrm{curl}(b)\times b)\ \dd x&=-\int_{\mathbb{R}^{3}} u_{2}\mathrm{div}\left(b^{\otimes 2}-\frac{1}{2}\left\vert b\right\vert^{2}I \right)\ \dd x+\underset{=0}{\underbrace{\int_{\mathbb{R}^{3}}(\mathrm{div}b)b\cdot u_{2}\ \dd x}} \notag\\
&=\int_{\mathbb{R}^{3}} D_{x}u_{2}:\left(b^{\otimes 2}-\frac{1}{2}\left\vert b\right\vert^{2} I\right)\ \dd x. \label{curl b times b}      
\end{align}
Substituting \eqref{e times curl e}-\eqref{curl b times b} inside \eqref{wedge product term} we get 
\begin{align*}
\mathcal{T}_{5}&=\int_{\mathbb{R}^{3}}D_{x}u_{2}:\left(e^{\otimes 2}-\frac{1}{2}\left\vert e\right\vert^{2}I+b^{\otimes 2}-\frac{1}{2}\left\vert b\right\vert^{2}I  \right)\ \dd x\\
&-\int_{\mathbb{R}^{3}}u_{2}\cdot \left((\rho_{2}u_{2}-\rho_{1}u_{1})\times b\right)\ \dd x
+\int_{\mathbb{R}^{3}}(\mathrm{div}e)e\cdot u_{2}\ \dd x. 
\end{align*}
Furthermore, we compute that   
\begin{align*}
\mathcal{T}_{3}&=\int_{\mathbb{R}^{3}} \rho_{1}u_{1}\cdot (u_{2}\times (B_{1}-B_{2}))\ \dd x\\
&=\int_{\mathbb{R}^{3}}\rho_{1} u_{2} \cdot ((B_{1}-B_{2})\times u_{1})\ \dd x=\int_{\mathbb{R}^{3}}\rho_{1}u_{2}\cdot (b\times u_{1})\ \dd x.     
\end{align*}
Consequently, it follows that 
\begin{align}
\mathcal{T}_{3}+\mathcal{T}_{5}&=\int_{\mathbb{R}^{3}}D_{x}u_{2}:\left(e^{\otimes 2}-\frac{1}{2}\left\vert e\right\vert ^{2}I+b^{\otimes 2}-\frac{1}{2}\left\vert b\right\vert^{2}I\right)\ \dd x
 \notag\\
&+\int_{\mathbb{R}^{3}}\rho_{1}u_{2}\cdot (b\times u_{1})\ \dd x-\int_{\mathbb{R}^{3}} u_{2}\cdot ((\rho_{2}u_{2}-\rho_{1}u
_{1})\times b)\ \dd x+\int_{\mathbb{R}^{3}}(\mathrm{div}e)e\cdot u_{2}\ \dd x \notag\\
&=\int_{\mathbb{R}^{3}}D_{x}u_{2}:\left(e^{\otimes 2}-\frac{1}{2}\left\vert e\right\vert^{2} I+b^{\otimes 2}-\frac{1}{2}\left\vert b\right\vert^{2}I \right)\ \dd x \notag\\
&+\underset{=0}{\underbrace{\int_{\mathbb{R}^{3}}\rho_{1}u_{2}\cdot (b\times u_{1})\ \dd x+\int_{\mathbb{R}^{3}}\rho_{1}u_{2}\cdot (u_{1}\times b)\ \dd x}}+\int_{\mathbb{R}^{3}}(\mathrm{div}e)e\cdot u_{2}\ \dd x \notag\\
&=\int_{\mathbb{R}^{3}}D_{x}u_{2}:\left(e^{\otimes 2}-\frac{1}{2}\left\vert e\right\vert^{2}I+b^{\otimes 2}-\frac{1}{2}\left\vert b\right\vert^{2}I  \right)\ \dd x+\int_{\mathbb{R}^{3}}(\mathrm{div}e)e\cdot u_{2}\ \dd x. \label{sum of T3T5} \end{align}
Finally, using that $\mathrm{div}E_{i}=\rho_{i}$ we see that  
\begin{align}
\mathcal{T}_{2}=-\int_{\mathbb{R}^{3}} (\rho_{1}-\rho_{2}) u_{2}\cdot (E_{1}-E_{2})\ \dd x=-\int_{\mathbb{R}^{3}}(\mathrm{div}e)e\cdot u_{2}\ \dd x 
\label{identity} 
\end{align}
and hence 
\begin{align}
\mathcal{T}_{2}+\mathcal{T}_{3}+\mathcal{T}_{5}=\int_{\mathbb{R}^{3}}D_{x}u_{2}:\left(e^{\otimes 2}-\frac{1}{2}\left\vert e\right\vert^{2}I+b^{\otimes 2}-\frac{1}{2}\left\vert b\right\vert^{2}I  \right)\ \dd x. \label{T2T3T5 classical}    
\end{align}
Substituting \eqref{T2T3T5 classical} inside \eqref{formula for time der of G} we get 
\begin{align*}
\frac{\dd}{\dd t}\mathcal{G}(t)=&-\int_{\mathbb{R}^{3}}\rho_{1}(u_{1}-u_{2})^{\otimes 2}:D_{x}u_{2}\ \dd x\\
&+\int_{\mathbb{R}^{3}}D_{x}u_{2}:\left(e^{\otimes 2}-\frac{1}{2}\left\vert e\right\vert^{2}I+b^{\otimes 2}-\frac{1}{2}\left\vert b\right\vert^{2}I  \right)\ \dd x\\
&-\int_{\mathbb{R}^{3}}\partial_{t}u_{2}\cdot ((E_{1}-E_{2})\times (B_{1}-B_{2}))\ \dd x\coloneqq \sum_{k=1}^{3}I_{k},  
\end{align*}
which establishes the announced identity \eqref{eq:time-der-G}.\\ 
\textbf{The estimate \eqref{H bound statement}.}  First note that 
\begin{align*}
I_{1}\leq \left\Vert D_{x}u_{2}\right\Vert_{\infty}\int_{\mathbb{R}^{3}}\rho_{1}\left\vert u_{1}-u_{2}\right\vert^{2}\ \dd x\leq 2\left\Vert D_{x}u_{2}\right\Vert_{\infty}\mathcal{U}(t) .      
\end{align*}
In addition, we have 
\begin{align*}
I_{2}&\leq \frac{3}{2}\left\Vert D_{x}u_{2}\right\Vert_{\infty}\int_{\mathbb{R}^{3}}\left\vert E_{1}-E_{2}\right\vert^{2}\ \dd x+\frac{3}{2}\left\Vert D_{x}u_{2}\right\Vert_{\infty}\int_{\mathbb{R}^{3}}\left\vert B_{1}-B_{2}\right\vert^{2}\ \dd x\\
&= 3\left\Vert D_{x}u_{2}\right\Vert_{\infty} (\mathcal{E}(t)+\mathcal{B}(t)). \end{align*}
Finally, we estimate that 
\begin{align*}
I_{3}\leq \left\Vert \partial_{t}u_{2}\right\Vert_{\infty}\left(\frac{1}{2}\int_{\mathbb{R}^{3}}\left\vert E_{1}-E_{2}\right\vert^{2}\ \dd x+\frac{1}{2}\int_{\mathbb{R}^{3}}\left\vert B_{1}-B_{2}\right\vert^{2}\ \dd x\right)\leq \left\Vert \partial_{t}u_{2}\right\Vert_{\infty}(\mathcal{E}(t)+\mathcal{B}(t)). \end{align*}
Gathering our inequalities we deduce that 
\begin{align*}
\frac{\dd}{\dd t}\mathcal{G}(t)\leq C\mathcal{G}(t)    
\end{align*}
so that 
\begin{align*}
\mathcal{G}(t)\leq \mathcal{G}(0)+C\int_{0}^{t}\mathcal{G}(\tau)\ \dd \tau.    
\end{align*}
By means of Lemma \ref{sandwich lemma} we find 
\begin{align*}
(1-\left\Vert u_{2}\right\Vert_{\infty})\mathcal{H}(t)\leq (1+\left\Vert u_{2}^{\mathrm{in}}\right\Vert_{\infty} )\mathcal{H}(0)+C(1+\left\Vert u_{2}\right\Vert_{\infty} )\int_{0}^{t}\mathcal{H}(\tau)\ \dd \tau,     
\end{align*}
or equivalently 
\begin{align*}
\mathcal{H}(t)\leq \frac{1+\left\Vert u^{\mathrm{in}}_{2}\right\Vert_{\infty} }{1-\left\Vert u_{2}\right\Vert_{\infty}}\mathcal{H}(0)+C\frac{1+\left\Vert u_{2}\right\Vert_{\infty}}{1-\left\Vert u_{2}\right\Vert_{\infty}}\int_{0}^{t}\mathcal{H}(\tau)\ \dd \tau.    
\end{align*}
The inequality \eqref{H bound statement} now follows immediately by Gr\"onwall's lemma. 
\qed
\section{The semiclassical limit: Proof of Theorem \ref{main thm}}\label{main thm section}
We abbreviate by $H_{\hbar}$ the magnetized self consistent quantum Hamiltonian given by  $H_{\hbar}\coloneqq \frac{1}{2}(i\hbar\nabla +A_{\hbar})^{2}+\varphi_{\hbar}$ so that the first equation in \eqref{Maxwell Schrodinger eq} takes the compact form  
\begin{align}
i\hbar \partial_{t}R_{\hbar}=\left[H_{\hbar},R_{\hbar}\right]. \label{equation in terms of Hamiltonian}     
\end{align}
Also, we will often use the fact that $\rho_{\hbar}$ is governed by the transport equation
\begin{align}
\partial_{t}\rho_{\hbar}+\mathrm{div}J_{\hbar}=0. \label{transport equation for quantum density}    
\end{align}
We start by establishing conservation of total energy:  
\begin{lem}\label{Conservation of energy}
Let the assumptions of Theorem \ref{main thm} hold. Let the total energy $\mathcal{F}_{\hbar}(t)$ be defined by 
\begin{align}
\mathcal{F}_{\hbar}(t)\coloneqq \frac{1}{2}\mathrm{tr}\left((i\hbar \nabla+A_{\hbar})^{2}R_{\hbar}\right)+\frac{1}{2}\int_{\mathbb{R}^{3}}\left\vert E_{\hbar}\right\vert^{2}\ \dd x+\frac{1}{2}\int_{\mathbb{R}^{3}}\left\vert B_{\hbar}\right\vert^{2}\ \dd x.       
\end{align}
Then, it holds that $\frac{\dd}{\dd t}\mathcal{F}_{\hbar}(t)=0$. 
\end{lem}
\begin{proof}
We compute that 
\begin{align}
&\frac{\dd}{\dd t}\frac{1}{2}\mathrm{tr}\left((i\hbar \nabla+A_{\hbar})^{2}R_{\hbar}\right)
=\frac{1}{2}\mathrm{tr}\left((i\hbar\nabla+A_{\hbar})\vee \partial_{t}A_{\hbar}R_{\hbar}\right) \notag\\
&-\frac{i}{2\hbar}\mathrm{tr}\left((i\hbar\nabla +A_{\hbar})^{2}\left[\frac{1}{2}(i\hbar \nabla+A_{\hbar})^{2}+\varphi_{\hbar},R_{\hbar}\right]\right) \notag\\
&=-\int_{\mathbb{R}^{3}}J_{\hbar}\cdot \partial_{t}A_{\hbar}\ \dd x+\frac{i}{2\hbar}\mathrm{tr}\left(\left[\frac{1}{2}(i\hbar\nabla+A_{\hbar})^{2}+\varphi_{\hbar},(i\hbar\nabla+A_{\hbar})^{2}\right]R_{\hbar}\right). \label{time der of kinetic part}
\end{align}
We have 
\begin{align*}
&\frac{i}{2\hbar}\left[\frac{1}{2}(i\hbar\nabla+A_{\hbar})^{2},(i\hbar \nabla+A_{\hbar})^{2}\right]\\
&=\sum_{1\leq j,k\leq 3}\frac{i}{4\hbar}(i\hbar\partial_{x_{k}}+A^{k}_{\hbar})\vee \left((i\hbar\partial_{x_{j}}+A^{j}_{\hbar})\vee \left[i\hbar\partial_{x_{j}}+A_{\hbar}^{j},i\hbar \partial_{x_{k}}+A_{\hbar}^{k}\right]\right)\\
&=\frac{i}{4\hbar}\sum_{1\leq j,k\leq 3}(i\hbar\partial_{x_{k}}+A^{k}_{\hbar})\vee \left((i\hbar\partial_{x_{j}}+A^{j}_{\hbar})\vee \left(i\hbar \partial_{x_{j}}A^{k}_{\hbar}-i\hbar\partial_{x_{k}}A^{j}_{\hbar}\right)\right)\\
&=\frac{1}{4}\sum_{1\leq j,k\leq 3}(i\hbar\partial_{x_{k}}+A^{k}_{\hbar})\vee \left((i\hbar\partial_{x_{j}}+A^{j}_{\hbar})\vee (\partial_{x_{k}}A^{j}_{\hbar}-\partial_{x_{j}}A^{k}_{\hbar})\right)=0.\end{align*}
The last equation  follows from the fact that if $\mathbf{J}$ is a antisymmetric matrix, $A:\mathbb{R}^{3}\rightarrow \mathbb{R}^{3}$ a given vector field  and $\Pi^{k}=i\hbar\partial_{x_{k}}+A^{k}$ then $\sum_{1\leq j,k\leq 3}\Pi^{k}\vee (\Pi^{j}\vee \mathbf{J}_{jk})=0$ (see claim 3.10 in \cite{ben2024magnetic}).   
In addition, it holds that 
\begin{align*}
\frac{i}{2\hbar}\left[\varphi_{\hbar},(i\hbar\nabla +A_{\hbar})^{2}\right]=-\frac{i}{2\hbar}(i\hbar \nabla+A_{\hbar})\vee (i\hbar \nabla \varphi_{\hbar})=\frac{1}{2} (i\hbar\nabla+A_{\hbar})\vee \nabla \varphi_{\hbar}    
\end{align*}
so that
\begin{align*}
\frac{i}{2\hbar}\mathrm{tr}\left(\left[\varphi_{\hbar},(i\hbar\nabla+A_{\hbar})^{2}\right]R_{\hbar}\right)=\frac{1}{2}\mathrm{tr}\left((i\hbar\nabla+A_{\hbar})\vee \nabla \varphi_{\hbar}R_{\hbar}\right)=-\int_{\mathbb{R}^{3}}\nabla \varphi_{\hbar}\cdot J_{\hbar}\ \dd x.     
\end{align*}
Thus, owing to \eqref{time der of kinetic part} we conclude that 
\begin{align}
\frac{\dd}{\dd t}\frac{1}{2}\mathrm{tr}\left((i\hbar \nabla+A_{\hbar})^{2}R_{\hbar}\right)=-\int_{\mathbb{R}^{3}} J_{\hbar}\cdot \partial_{t}A_{\hbar}\ \dd x-\int_{\mathbb{R}^{3}}J_{\hbar}\cdot\nabla \varphi_{\hbar}\ \dd x= \int_{\mathbb{R}^{3}}J_{\hbar}\cdot E_{\hbar}\ \dd x. 
\label{time der of quantum nonmodulated kinetic part}
\end{align}
In addition, we have 
\begin{align}
&\frac{\dd}{\dd t}\frac{1}{2}\left(\int_{\mathbb{R}^{3}}\left\vert E_{\hbar}\right\vert^{2}\ \dd x+\int_{\mathbb{R}^{3}}\left\vert B_{\hbar}\right\vert^{2}\ \dd x  \right) \notag\\
&=\int_{\mathbb{R}^{3}}E_{\hbar}\cdot\partial_{t}E_{\hbar}\ \dd x+\int_{\mathbb{R}^{3}}B_{\hbar}\cdot \partial_{t}B_{\hbar}\ \dd x \notag\\
&=\int_{\mathbb{R}^{3}}E_{\hbar}\cdot\partial_{t}E_{\hbar}\ \dd x+\int_{\mathbb{R}^{3}}\mathrm{curl}(A_{\hbar})\cdot \partial_{t}\mathrm{curl}(A_{\hbar})\ \dd x \notag\\
&=\int_{\mathbb{R}^{3}}E_{\hbar}\cdot \partial_{t}E_{\hbar}\ \dd x+\int_{\mathbb{R}^{3}}\mathrm{curl}^{2}(A_{\hbar})\cdot \partial_{t}A_{\hbar}\ \dd x \notag\\
&=\int_{\mathbb{R}^{3}}E_{\hbar}\cdot \partial_{t}E_{\hbar}\ \dd x+\int_{\mathbb{R}^{3}}\Delta A_{\hbar}\cdot (\nabla \varphi_{\hbar}+E_{\hbar})\ \dd x, \label{time der of nonmodulated maxwell part}
\end{align}
where in last equation we used the identity $\mathrm{curl}^
{2}(A)=-\Delta A$ for divergence-free $A$. Furthermore, using that $\partial_{t}\rho_{\hbar}+\mathrm{div}J_{\hbar}=0$ (by \eqref{transport equation for quantum density}) we see that  
\begin{align}
\partial_{t}E_{\hbar}=-\partial_{tt}A_{\hbar}-\partial_{t}\nabla \varphi_{\hbar}=-\Delta A_{\hbar}-\mathbb{P}J_{\hbar}-\nabla(\Delta)^{-1}\mathrm{div}J_{\hbar}=-\Delta A_{\hbar}-J_{\hbar}.\label{time der of Eh}
\end{align}
Therefore the right-hand side of \eqref{time der of nonmodulated maxwell part} is 
\begin{align}
-\int_{\mathbb{R}^{3}}E_{\hbar}\cdot (\Delta A_{\hbar}+J_{\hbar})\ \dd x+\int_{\mathbb{R}^{3}}\Delta A_{\hbar}\cdot E_{\hbar}\ \dd x+\underset{=0}{\underbrace{\int_{\mathbb{R}^{3}}\Delta A_{\hbar}\cdot \nabla \varphi_{\hbar}\ \dd x} }=-\int_{\mathbb{R}^{3}}E_{\hbar}\cdot J_{\hbar}\ \dd x.  \label{time der of quantum nonmod maxwell part final} 
\end{align}
The combination of \eqref{time der of quantum nonmodulated kinetic part} and \eqref{time der of quantum nonmod maxwell part final} yields $\frac{\dd}{\dd t}\mathcal{F}_{\hbar}(t)=0$. 
\end{proof}
\begin{thm} 
Let the assumptions of Theorem \ref{main thm} hold and let $\mathcal{H}_{\hbar}(t)$ be given by \eqref{def of quantum modualted energy}.
Then, it holds that 
\begin{align}
\frac{\dd}{\dd t}\mathcal{H}_{\hbar}(t)=&-\frac{1}{4}\sum_{1\leq j,k\leq 3}\mathrm{tr}\left((i\hbar\partial_{x_{j}}+A^{j}_{\hbar}+u^{j})\vee \left((i\hbar\partial_{x_{k}}+A^{k}_{\hbar}+u^{k})\vee \partial_{x_{k}}u^{j}\right)R_{\hbar}\right)\notag\\
&-\int_{\mathbb{R}^{3}}(\rho_{\hbar}-\rho)u\cdot (E_{\hbar}-E)\ \dd x \notag\\
&+\int_{\mathbb{R}^{3}}(J_{\hbar}-\rho_{\hbar}u)\cdot (u\times (B_{\hbar}-B))\ \dd x.  \label{time der of modulated energy statement}
\end{align}
\label{time der of quantum modulated energy thm}
\end{thm}
\begin{proof}
\textbf{Calculation of $\frac{\dd}{\dd t}\mathcal{U}_{\hbar}(t)$.}  We claim to have the following formula 
\begin{align}
\frac{\dd}{\dd t}\mathcal{U}_{\hbar}(t)&=-\frac{1}{4}\sum_{1\leq j,k\leq 3}\mathrm{tr}\left((i\hbar\partial_{x_{j}}+A^{j}_{\hbar}+u^{j})\vee\left(( i\hbar\partial_{x_{k}}+A^{k}_{\hbar}+u^{k})\vee  \partial_{x_{j}}u^{k}\right)R_{\hbar}\right) \notag\\
&+\int_{\mathbb{R}^{3}}(J_{\hbar}-\rho_{\hbar}u)\cdot (E_{\hbar}-E)\ \dd x \notag\\
&+\int_{\mathbb{R}^{3}}(J_{\hbar}-\rho_{\hbar}u)\cdot (u\times (B_{\hbar}-B))\ \dd x. \label{time der of U} 
\end{align}
By expanding we see that \begin{align*}
\frac{\dd}{\dd t}\mathcal{U}_{\hbar}(t)&=\frac{1}{2}\frac{\dd}{\dd t}\mathrm{tr}\left((i\hbar \nabla+A_{\hbar})^{2}R_{\hbar}\right)+\frac{1}{2}\frac{\dd}{\dd t}\mathrm{tr}\left((i\hbar \nabla +A_{\hbar})\vee u R_{\hbar}\right)+\frac{\dd}{\dd t}\frac{1}{2}\int_{\mathbb{R}^{3}}\left\vert u\right\vert^{2}\rho_{\hbar}\ \dd x\coloneqq\sum_{k=1}^{3}I_{k}.\end{align*}
In the sequel we shall freely make use of \eqref{equation in terms of Hamiltonian}. We compute:
\begin{align*}
I_{2}+I_{3}&=\frac{\dd}{\dd t}\frac{1}{2}\mathrm{tr}\left((i\hbar \nabla +A_{\hbar})\vee uR_{\hbar}\right)+\frac{1}{2}\frac{\dd}{\dd t}\mathrm{tr}\left(\left\vert u\right\vert^{2}R_{\hbar} \right)\\
&=\frac{1}{2}\mathrm{tr}\left(((i\hbar \nabla +A_{\hbar})\vee u+\left\vert u\right\vert^{2})\partial_{t}R_{\hbar} \right)+\frac{1}{2}\mathrm{tr}\left(\partial_{t}\left((i\hbar \nabla+A_{\hbar})\vee u+\left\vert u\right\vert^{2} \right)R_{\hbar}\right)
\\
&=\frac{1}{2i\hbar}\mathrm{tr}\left(((i\hbar \nabla +A_{\hbar})\vee u+\left\vert u\right\vert^{2} )\left[H_{\hbar},R_{\hbar}\right]\right)+\frac{1}{2}\mathrm{tr}\left(\partial_{t}\left((i\hbar \nabla+A_{\hbar})\vee u+\left\vert u\right\vert^{2} \right)R_{\hbar}\right)\\
&=\frac{i}{2\hbar}\mathrm{tr}\left(\left[H_{\hbar},(i\hbar\nabla+A_{\hbar})\vee u+\left\vert u\right\vert^{2} \right]R_{\hbar}\right)+\frac{1}{2}\mathrm{tr}\left(\partial_{t}\left((i\hbar\nabla+A_{\hbar})\vee u+\left\vert u\right\vert^{2} \right)R_{\hbar}\right)\\
&=\frac{1}{2}\mathrm{tr}\left(\left(\partial_{t}+\frac{i}{\hbar}\left[H_{\hbar},\cdot\right]\right)\left((i\hbar\nabla+A_{\hbar})\vee u+\left\vert u\right\vert^{2} \right)R_{\hbar}\right)\\
&=\frac{1}{2}\mathrm{tr}\left(\left(\partial_{t}+\frac{i}{\hbar}\left[\frac{1}{2}(i\hbar\nabla+A_{\hbar})^{2},\cdot\right]\right)\left((i\hbar\nabla+A_{\hbar})\vee u+\frac{1}{2}u\vee u\right)R_{\hbar}\right)\\
&+\frac{1}{2}\mathrm{tr}\left(\frac{i}{\hbar}\left[\varphi_{\hbar},\cdot\right]\left((i\hbar\nabla+A_{\hbar})\vee u+\frac{1}{2}u\vee u\right)R_{\hbar}\right)\\
&=\frac{1}{2}\sum_{j=1}^{3}\mathrm{tr}\left(\left(\partial_{t}+\frac{i}{\hbar}\left[\frac{1}{2}(i\hbar\nabla+A_{\hbar})^{2},\cdot\right]\right)\left((i\hbar\partial_{x_{j}}+A_{\hbar}^{j})\vee u^{j}+\frac{1}{2}u^{j}\vee u^{j}\right)R_{\hbar}\right)\\
&+\frac{1}{2}\sum_{j=1}^{3}\mathrm{tr}\left(\frac{i}{\hbar}\left[\varphi_{\hbar},\cdot\right]\left((i\hbar\partial_{x_{j}}+A_{\hbar}^{j})\vee u^{j}+\frac{1}{2}u^{j}\vee u^{j}\right)R_{\hbar}\right)\\
&\coloneqq\frac{1}{2}\sum_{j=1}^{3}\mathrm{tr}\left(L_{1j}R_{\hbar}\right)+\frac{1}{2}\sum_{j=1}^{3}\mathrm{tr}\left(L_{2j}R_{\hbar}\right).  
\end{align*}
Set $\Pi^{j}\coloneqq i\hbar\partial_{x_{j}}+A^{j}_{\hbar}$. We start with the term $L_{1j}$. Using the  Leibniz rule: $\left[A,B\vee C\right]=\left[A,B\right]\vee C+\left[A,C\right]\vee B,$ we can write $L_{1j}$ as follows: 
\begin{align}
L_{1j}&=\left(\partial_{t}+\frac{i}{\hbar}\left[\frac{1}{2}\sum_{k=1}^{3}(\Pi^{k})^{2},\cdot\right]\right)\left((\frac{1}{2}u^{j}+\Pi^{j})\vee u^{j}\right) \notag\\
&=\left(\left(\partial_{t}+\frac{i}{\hbar}\left[\frac{1}{2}\sum_{k=1}^{3}(\Pi^{k})^{2},\cdot\right]\right)(\frac{1}{2}u^{j}+\Pi^{j})\right)\vee u^{j} \notag\\
&+(\frac{1}{2}u^{j}+\Pi^{j})\vee \left(\left(\partial_{t}+\frac{i}{\hbar}\left[\frac{1}{2}\sum_{k=1}^{3}(\Pi^{k})^{2},\cdot\right]\right)u^{j}\right) \notag\\
&=\Big(\Big(\partial_{t}+\frac{i}{\hbar}\Big[\frac{1}{2}\sum_{k=1}^{3}(\Pi^{k})^{2},\cdot\Big]\Big)u^{j}\Big)\vee u^{j}+\Pi^{j}\vee \Big(\Big(\partial_{t}+\frac{i}{\hbar}\Big[\frac{1}{2}\sum_{k=1}^{3}(\Pi^{k})^{2},\cdot\Big]\Big)u^{j}\Big) \notag\\
&+\left(\left(\partial_{t}+\frac{i}{\hbar}\Big[\frac{1}{2}\sum_{k=1}^{3}(\Pi^{k})^{2},\cdot\Big]\right)\Pi^{j}\right)\vee u^{j}. \label{L1j first formula} 
\end{align}
We compute that 
\begin{align*}
\left[(\Pi^{k})^{2},\Pi^{j}\right]=\Pi^{k}\vee\left[\Pi^{k},\Pi^{j}\right]=\Pi^{k}\vee \left[i\hbar\partial_{x_{k}}+A^{k}_{\hbar},i\hbar \partial_{x_{j}}+A^{j}_{\hbar}\right]=\Pi^{k}\vee (i\hbar \partial_{x_{k}}A^{j}_{\hbar}-i\hbar\partial_{x_{j}}A ^{k}_{\hbar}), 
\end{align*}
and hence setting $\mathbf{J}_{jk}=\partial_{x_{j}}A^{k}_{\hbar}-\partial_{x_{k}}A^{j}_{\hbar}$ we get   
\begin{align}
\frac{i}{\hbar}\Big[\frac{1}{2}\sum_{k=1}^{3}(\Pi^{k})^{2},\Pi^{j}\Big]=\frac{1}{2}\sum_{k=1}^{3}\Pi^{k}\vee(\partial_{x_{j}}A^{k}_{\hbar}-\partial_{x_{k}}A^{j}_{\hbar})=\frac{1}{2}\sum_{k=1}^{3}\Pi^{k}\vee \mathbf{J}_{jk}.  \label{last term 1body}     
\end{align}
Inserting \eqref{last term 1body} inside \eqref{L1j first formula} we get 
\begin{align}
L_{1j}&=(u^{j}+\Pi^{j})\vee \Big(\Big(\partial_{t}+\frac{i}{\hbar}\left[\frac{1}{2}\sum_{k=1}^{3}(\Pi^{k})^{2},\cdot\right]\Big)u^{j}\Big)+\partial_{t}A^{j}_{\hbar}\vee u^{j}+\frac{1}{2}\sum_{k=1}^{3}(\Pi^{k}\vee \mathbf{J}_{jk})\vee u^{j} \notag\\
&=(u^{j}+\Pi^{j})\vee \Big(\partial_{t}u^{j}+\sum_{k=1}^{3}u^{k}\partial_{x_{k}}u^{j}\Big)+(u^{j}+\Pi^{j})\vee \Big(\frac{i}{\hbar}\Big[\frac{1}{2}\sum_{k=1}^{3}(\Pi^{k})^{2},u^{j}\Big]-\sum_{k=1}^{3}u^{k}\partial_{x_{k}}u^{j}\Big) \notag\\
&+\frac{1}{2}\sum_{k=1}^{3}(\Pi^{k}\vee \mathbf{J}_{jk})\vee u^{j}+\partial_{t}A_{\hbar}^{j}\vee u^{j}. \label{second right hand side L1j} 
\end{align}
Note that 
\begin{align*}
\frac{i}{2\hbar}\left[(\Pi^{k})^{2},u
^{j}\right]=\frac{i}{2\hbar}\Pi^{k}\vee \left[i\hbar\partial_{x_{k}},u^{j}\right]=-\frac{1}{2}\Pi^{k}\vee \partial_{x_{k}}u^{j}
\end{align*}
so that the second term in the right-hand side of \eqref{second right hand side L1j} reads
\begin{align*}
&-\sum_{k=1}^{3}(u^{j}+\Pi^{j})\vee (\frac{1}{2}\Pi^{k}\vee \partial_{x_{k}}u^{j}+u^{k}\partial_{x_{k}}u^{j})\\
&=-\sum_{k=1}^{3}(u^{j}+\Pi^{j})\vee (\frac{1}{2}\Pi^{k}\vee \partial_{x_{k}}u^{j}+\frac{1}{2}u^{k}\vee \partial_{x_{k}}u^{j})\\
&=-\frac{1}{2}\sum_{k=1}^{3}(u^{j}+\Pi^{j})\vee ((u^{k}+\Pi^{k})\vee \partial_{x_{k}}u^{j}).     
\end{align*}
So we find that 
\begin{align*}
L_{1j}=&(u^{j}+\Pi^{j})\vee \Big(\partial_{t}u^{j}+\sum_{k=1}^{3}u^{k}\partial_{x_{k}}u^{j}\Big)-\frac{1}{2}\sum_{k=1}^{3}(u^{j}+\Pi^{j})\vee \left((u^{k}+\Pi^{k})\vee \partial_{x_{k}}u^{j}\right)\\
+&\frac{1}{2}\sum_{k=1}^{3}(\Pi^{k}\vee \mathbf{J}_{jk})\vee u^{j}+\partial_{t}A_{\hbar}^{j}\vee u^{j}. 
\end{align*}
Therefore we conclude that
\begin{align}
\frac{1}{2}\sum_{j=1}^{3}\mathrm{tr}\big(L_{1j}R_{\hbar}\big)
&=\frac{1}{2}\sum_{j=1}^{3}\mathrm{tr}\big((u^{j}+\Pi^{j})\vee(\partial_{t}u^{j}+(uD_{x}u )^{j})R_{\hbar}\big) \notag \\
&-\frac{1}{4}\sum_{1\leq j,k\leq 3}\mathrm{tr}\big((u^{j}+\Pi^{j})\vee((u^{k}+\Pi^{k})\vee\partial_{x_{k}}u^{j})R_{\hbar}\big) \notag\\
&+\frac{1}{2}\sum_{j=1}^{3}\mathrm{tr}\left(\left(\frac{1}{2}\sum_{k=1}^{3}(\Pi^{k}\vee \mathbf{J}_{jk})\vee u^{j}+\partial_{t}A^{j}_{\hbar}\vee u^{j}\right) R_{\hbar}\right).
\label{eq for J1 1body}
\end{align}
As for $L_{2j}$, note the following identities
\begin{align*}
[\varphi_{\hbar},u^{j}\vee u^{j}]=0,\qquad [\varphi_{\hbar},A^{j}_{\hbar}u^{j}]=0, \qquad  
-[\varphi_{\hbar},\partial_{x_{j}}\vee u^{j}]=-\left[\varphi_{\hbar},2u^{j}\partial_{x_{j}}\right]=2u^{j}\partial_{x_{j}}\varphi_{\hbar}.
\end{align*}
Consequently, it follows that 
\begin{align}
\frac{1}{2}\sum_{j=1}^{3}\mathrm{tr}\big(L_{2j}R_{\hbar}\big)=\int_{\mathbb{R}^{3}}\rho_{\hbar}\nabla \varphi_{\hbar}\cdot u\,\dd x.
\label{eq for J2 1body}
\end{align}
Therefore we infer that 
\begin{align}
I_{2}+I_{3}&= \frac{1}{2}\sum_{j=1}^{3}\mathrm{tr}\big((u^{j}+\Pi^{j})\vee(\partial_{t}u^{j}+(uD_{x}u )^{j})R_{\hbar}\big) \notag \\
&-\frac{1}{4}\sum_{1\leq j,k\leq 3}\mathrm{tr}\big((u^{j}+\Pi^{j})\vee((u^{k}+\Pi^{k})\vee\partial_{x_{k}}u^{j})R_{\hbar}\big) \notag\\
&+\frac{1}{2}\sum_{j=1}^{3}\mathrm{tr}\left(\left(\frac{1}{2}\sum_{k=1}^{3}(\Pi^{k}\vee \mathbf{J}_{jk})\vee u^{j}+\partial_{t}A^{j}_{\hbar}\vee u^{j}\right) R_{\hbar}\right) \notag\\
&+\int_{\mathbb{R}^{3}}\rho_{\hbar}\nabla \varphi_{\hbar}\cdot u\,\dd x\coloneqq T_{1}+T_{2}+T_{3}+T_{4}. \label{T2 formula}\end{align}
Due to \eqref{Euler Maxwell} we have 
\begin{align}
T_{1}&=\frac{1}{2}\sum_{j=1}^{3}\mathrm{tr}\left((u^{j}+\Pi^{j})\vee (E+u\times B)^{j}R_{\hbar}\right) \notag\\
&=\int_{\mathbb{R}^{3}}\rho_{\hbar}u\cdot E\ \dd x-\int_{\mathbb{R}^{3}}J_{\hbar}\cdot (E+u\times B)\ \dd x \notag\\
&=\int_{\mathbb{R}^{3}}(\rho_{\hbar}u-J_{\hbar})\cdot E\ \dd x-\int_{\mathbb{R}^{3}}J_{\hbar}\cdot u\times B\ \dd x. 
\label{T1formula}
\end{align}
By Lemma 4.2 in \cite{porat2026magnetized} there holds the identity 
\begin{align*}
\sum_{1\leq j,k\leq 3}(\Pi^{k}\vee \mathbf{J}_{jk})\vee u^{j}=\sum_{1\leq j,k\leq 3}\Pi^{j}\vee (\mathbf{J}_{kj}\vee u^{k}),    
\end{align*}
and therefore the term $T_{3}$ is recast as  
\begin{align}
T_{3}&=\frac{1}{4}\mathrm{tr}\left(\sum_{1\leq j,k\leq 3} \Pi^{j}\vee (\mathbf{J}_{kj}\vee u^{k})R_{\hbar}\right)+\int_{\mathbb{R}^{3}}\partial_{t}A_{\hbar}\cdot u\rho_{\hbar}\ \dd x \notag\\
&=-\sum_{1\leq j,k\leq 3}\int_{\mathbb{R}^{3}}J_{\hbar}^{j}\mathbf{J}_{kj}u^{k}\ \dd x+\int_{\mathbb{R}^{3}}\partial_{t}A_{\hbar}\cdot u\rho_{\hbar}\ \dd x \notag\\
&=\int_{\mathbb{R}^{3}}J_{\hbar}\cdot(u\times B_{\hbar})\ \dd x+\int_{\mathbb{R}^{3}}\partial_{t}A_{\hbar}\cdot u\rho_{\hbar}\ \dd x. \label{T3formula}
\end{align}
The term $T_{4}$ is recast as  
\begin{align}
T_{4}=-\int_{\mathbb{R}^{3}}\rho_{\hbar}(\partial_{t}A_{\hbar}+E_{\hbar})\cdot u\ \dd x. \label{T4formula}   
\end{align}
By Lemma \ref{Conservation of energy} (specifically equation \eqref{time der of quantum nonmodulated kinetic part}) we have  
\begin{align}
I_{1}=\int_{\mathbb{R}^{3}}J_{\hbar}\cdot E_{\hbar}\ \dd x. \label{I1 formula}
\end{align}
In view of \eqref{T1formula}-\eqref{I1 formula} we get 
\begin{align}
I_{1}+T_{1}+T_{3}+T_{4}
&= \int_{\mathbb{R}^{3}}J_{\hbar}\cdot E_{\hbar}\ \dd x+\int_{\mathbb{R}^{3}}J_{\hbar}\cdot (u\times B_{\hbar})\ \dd x-\int_{\mathbb{R}^{3}}\rho_{\hbar}u\cdot E_{\hbar}\ \dd x \notag\\
&+\int_{\mathbb{R}^{3}}(\rho_{\hbar}u-J_{\hbar})\cdot E\ \dd x-\int_{\mathbb{R}^{3}}J_{\hbar}\cdot u\times B\ \dd x \notag\\
&=\int_{\mathbb{R}^{3}}J_{\hbar}\cdot E_{\hbar}\ \dd x+\int_{\mathbb{R}^{3}}J_{\hbar}\cdot u\times (B_{\hbar}-B)\ \dd x \notag\\
&-\int_{\mathbb{R}^{3}}\rho_{\hbar}u\cdot E_{\hbar}\ \dd x+\int_{\mathbb{R}^{3}}(\rho_{\hbar}u-J_{\hbar})\cdot E\ \dd x \notag\\
&=\int_{\mathbb{R}^{3}}J_{\hbar}\cdot E_{\hbar}\ \dd x-\int_{\mathbb{R}^{3}}\rho_{\hbar}u\cdot E_{\hbar}\ \dd x+\int_{\mathbb{R}^{3}}(\rho_{\hbar}u-J_{\hbar})\cdot E\ \dd x \notag\\
&+\int_{\mathbb{R}^{3}}(J_{\hbar}-\rho_{\hbar}u)\cdot u\times (B_{\hbar}-B)\ \dd x \notag\\
&=\int_{\mathbb{R}^{3}}(J_{\hbar}-\rho_{\hbar}u)\cdot (E_{\hbar}-E)\ \dd x+\int_{\mathbb{R}^{3}}(J_{\hbar}-\rho_{\hbar}u)\cdot u\times (B_{\hbar}-B)\ \dd x.  \label{I1T1T3T4}
\end{align}
The combination of \eqref{I1T1T3T4} and \eqref{T2 formula} leads to the announced formula \eqref{time der of U}. \\
\textbf{Calculation of $\frac{\dd}{\dd t}\mathcal{E}_{\hbar}(t)$+$\frac{\dd}{\dd t}\mathcal{B}_{\hbar}(t)$.} We claim to have the formula 
\begin{align}
\frac{\dd}{\dd t}\mathcal{E}_{\hbar}(t)+\frac{\dd}{\dd t}\mathcal{B}_{\hbar}(t)=-\int_{\mathbb{R}^{3}}(E_{\hbar}-E)\cdot (J_{\hbar}-\rho u)\ \dd x. \label{time der of Maxwell part}    
\end{align}
Thanks to \eqref{time der of Eh}, we compute that 
\begin{align*}
\frac{\dd}{\dd t}\frac{1}{2}\int_{\mathbb{R}^{3}}\left\vert E_{\hbar}-E\right\vert^{2}\ \dd x&=\int_{\mathbb{R}^{3}} (\partial_{t}E_{\hbar}-\partial_{t}E)\cdot (E_{\hbar}-E)\ \dd x\\
&=\int_{\mathbb{R}^{3}}(-(\Delta A_{\hbar}+J_{\hbar})-(\mathrm{curl}(B)-\rho u))\cdot (E_{\hbar}-E)\ \dd x\\
&=-\int_{\mathbb{R}^{3}}J_{\hbar}\cdot (E_{\hbar}-E)\ \dd x-\int_{\mathbb{R}^{3}}(\Delta A_{\hbar}+\mathrm{curl}(B)-\rho u)\cdot (E_{\hbar}-E)\ \dd x.
\end{align*}
In addition, we have
\begin{align*}
\frac{\dd }{\dd t} \frac{1}{2}\int_{\mathbb{R}^{3}}\left\vert B_{\hbar}-B\right\vert^{2}\ \dd x&=\int_{\mathbb{R}^{3}}(\partial_{t}B_{\hbar}-\partial_{t}B)\cdot (B_{\hbar}-B)\ \dd x\\
&=\int_{\mathbb{R}^{3}}(\partial_{t}\mathrm{curl}(A_{\hbar})+\mathrm{curl}(E))\cdot (B_{\hbar}-B)\ \dd x\\
&=\int_{\mathbb{R}^{3}}(\partial_{t}A_{\hbar}+E)\cdot \mathrm{curl}(B_{\hbar}-B)\ \dd x\\
&=\int_{\mathbb{R}^{3}}(E-E_{\hbar}-\nabla \varphi_{\hbar})\cdot \mathrm{curl}(B_{\hbar}-B)\ \dd x\\
&=\int_{\mathbb{R}^{3}}(E_{\hbar}-E)\cdot \mathrm{curl}(B)\ \dd x+\int_{\mathbb{R}^{3}}\nabla \varphi_{\hbar}\cdot \mathrm{curl} B\ \dd x\\
&+\int_{\mathbb{R}^{3}}(E-E_{\hbar}-\nabla \varphi_{\hbar})\cdot \mathrm{curl}B_{\hbar}\ \dd x. 
\end{align*}
Therefore it follows that 
\begin{align*}
\frac{\dd}{\dd t}\mathcal{E}_{\hbar}(t)+\frac{\dd}{\dd t}\mathcal{B}_{\hbar}(t)&=-\int_{\mathbb{R}^{3}}J_{\hbar}\cdot (E_{\hbar}-E)\ \dd x-\int_{\mathbb{R}^{3}}(\Delta A_{\hbar}-\rho u)
\cdot (E_{\hbar}-E)\ \dd x\\
&+\int_{\mathbb{R}^{3}}\nabla \varphi_{\hbar}\cdot \mathrm{curl}B\ \dd x+\int_{\mathbb{R}^{3}}(E-E_{\hbar}-\nabla \varphi_{\hbar})\cdot \mathrm{curl}(B_{\hbar})\ \dd x. 
\end{align*}
Since $\mathrm{curl}(\nabla \varphi_{\hbar})=0$, integration by parts yields $\int_{\mathbb{R}^{3}}\nabla \varphi_{\hbar}\cdot \mathrm{curl} B\ \dd x=0$ so that we get 
\begin{align*}
\frac{\dd}{\dd t}\mathcal{E}_{\hbar}(t)+\frac{\dd}{\dd t}\mathcal{B}_{\hbar}(t)=-\int_{\mathbb{R}^{3}}(E_{\hbar}-E)\cdot (J_{\hbar}+\Delta A_{\hbar}-\rho u+\mathrm{curl}(B_{\hbar}))\ \dd x.    
\end{align*}
Considering that $\mathrm{curl}(B_{\hbar})=\mathrm{curl}^{2}A_{\hbar}=-\Delta A_{\hbar}$ we deduce that 
\begin{align*}
\frac{\dd}{\dd t}\mathcal{E}_{\hbar}(t)+\frac{\dd}{\dd t}\mathcal{B}_{\hbar}(t)=-\int_{\mathbb{R}^{3}} (E_{\hbar}-E)\cdot (J_{\hbar}-\rho u)\ \dd x,     
\end{align*}
which establishes \eqref{time der of Maxwell part}. \\
\textbf{Conclusion.} In view of \eqref{time der of U} and \eqref{time der of Maxwell part} we get 
\begin{align*}
 \frac{\dd}{\dd t}\mathcal{H}_{\hbar}(t)&=-\frac{1}{4}\sum_{1\leq j,k\leq 3}\mathrm{tr}\left((i\hbar\partial_{x_{j}}+A^{j}_{\hbar}+u^{j})\vee \left((i\hbar\partial_{x_{k}}+A^{k}_{\hbar}+u^{k})\vee  \partial_{x_{j}}u^{k}\right) R_{\hbar}\right) \notag\\
&+\int_{\mathbb{R}^{3}}(J_{\hbar}-\rho_{\hbar}u)\cdot (E_{\hbar}-E)\ \dd x \notag\\
&+\int_{\mathbb{R}^{3}}(J_{\hbar}-\rho_{\hbar}u)\cdot (u\times (B_{\hbar}-B))\ \dd x-\int_{\mathbb{R}^{3}} (E_{\hbar}-E)\cdot (J_{\hbar}-\rho u)\ \dd x \notag\\
=&-\frac{1}{4}\sum_{1\leq j,k\leq 3}\mathrm{tr}\left(\left(i\hbar\partial_{x_{j}}+A^{j}_{\hbar}+u^{j}\right)\vee \left((i\hbar \partial_{x_{k}}+A^{k}_{\hbar}+u^{k})\vee \partial_{x_{j}}u^{k}\right)R_{\hbar}\right) \notag\\
&-\int_{\mathbb{R}^{3}}(\rho_{\hbar}-\rho)u\cdot (E_{\hbar}-E)\ \dd x+\int_{\mathbb{R}^{3}}(J_{\hbar}-\rho_{\hbar}u)\cdot (u\times (B_{\hbar}-B))\ \dd x . 
\end{align*}
\end{proof}
\textbf{Proof of Theorem \ref{main thm}.} 
\textbf{The identity \eqref{Time derivative of quantum corrected modulated energy statement}.} 
Recall the abbreviation $e_{\hbar}=E_{\hbar}-E$ and $b_{\hbar}=B_{\hbar}-B$. By Theorem \ref{time der of quantum modulated energy thm} we have  
\begin{align}
\frac{\dd}{\dd t}\mathcal{G}_{\hbar}(t)=&-\frac{1}{4}\sum_{1\leq j,k\leq 3}\mathrm{tr}\left((i\hbar\partial_{x_{j}}+A^{j}_{\hbar}+u^{j})\vee \left((i\hbar\partial_{x_{k}}+A^{k}_{\hbar}+u^{k})\vee \partial_{x_{k}}u^{j}\right)R_{\hbar}\right)\notag
\\
&-\int_{\mathbb{R}^{3}}(\rho_{\hbar}-\rho)u\cdot e_{\hbar}\ \dd x \notag\\
&+\int_{\mathbb{R}^{3}}(J_{\hbar}-\rho_{\hbar}u)\cdot (u\times b_{\hbar})\ \dd x-\int_{\mathbb{R}^{3}}\partial_{t}u\cdot (e_{\hbar}\times b_{\hbar})\ \dd x \notag\\
&-\int_{\mathbb{R}^{3}}u\cdot (\partial_{t}e_{\hbar}\times b_{\hbar}+e_{\hbar}\times \partial_{t}b_{\hbar})\ \dd x= \sum_{j=1}^{5}\mathcal{T}_{j}. 
\label{formula for time der of quantum G}
\end{align}
Using \eqref{time der of Eh} we get  the following identities:  
\begin{align*}
\partial_{t}e_{\hbar}&=-\Delta A_{\hbar}-J_{\hbar}+\rho u-\mathrm{curl}(B)
= \rho u-J_{\hbar}+\mathrm{curl}(b_{\hbar}) 
\end{align*}
and 
\begin{align*}
\partial_{t}b_{\hbar}=\mathrm{curl}(E)-\mathrm{curl}(E_{\hbar})=-\mathrm{curl}(e_{\hbar}).     
\end{align*}
Therefore, it follows that  
\begin{align}
\mathcal{T}_{5}=
\int_{\mathbb{R}^{3}} u\cdot (e_{\hbar}\times \mathrm{curl}(e_{\hbar}))\ \dd x
-\int_{\mathbb{R}^{3}} u\cdot (\mathrm{curl}(b_{\hbar})\times b_{\hbar})\ \dd x
-\int_{\mathbb{R}^{3}} u\cdot ((\rho u-J_{\hbar})\times b_{\hbar})\ \dd x. \label{wedge quantum product term}
\end{align}
Recall that for a vector field $v:\mathbb{R}^{3}\rightarrow \mathbb{R}^{3}$ there holds the identity 
\begin{align*}
\mathrm{curl}(v)\times v=\mathrm{div}\left(v
^{\otimes 2} -\frac{1}{2}\left\vert v\right\vert^{2} I\right)-(\mathrm{div}v)v.   
\end{align*}
Therefore, we find that 
\begin{align}
\int_{\mathbb{R}^{3}} u\cdot (e_{\hbar}\times \mathrm{curl}(e_{\hbar}))\ \dd x&=-\int_{\mathbb{R}^{3}} u\cdot \mathrm{div}\left(e^{\otimes 2}_{\hbar}-\frac{1}{2}\left\vert e_{\hbar}\right\vert^{2}I \right)\ \dd x+\int_{\mathbb{R}^{3}}u\cdot e_{\hbar}\mathrm{div}e_{\hbar}\ \dd x \notag\\
&=\int_{\mathbb{R}^{3}} D_{x}u:\left(e^{\otimes 2}_{\hbar} -\frac{1}{2}\left\vert e_{\hbar}\right\vert^{2}I \right)\ \dd x+\int_{\mathbb{R}^{3}}u\cdot e_{\hbar}\mathrm{div}e_{\hbar}\ \dd x  
\label{eh times curl eh}
\end{align}
and that 
\begin{align}
-\int_{\mathbb{R}^{3}} u\cdot (\mathrm{curl}(b_{\hbar})\times b_{\hbar})\ \dd x&=-\int_{\mathbb{R}^{3}} u\cdot \mathrm{div}\left(b^{\otimes 2}_{\hbar}-\frac{1}{2}\left\vert b_{\hbar}\right\vert^{2}I \right)\ \dd x+\underset{=0}{\underbrace{\int_{\mathbb{R}^{3}}u\cdot b_{\hbar}\mathrm{div}b_{\hbar}\ \dd x}} \notag\\
&=\int_{\mathbb{R}^{3}} D_{x}u:\left(b^{\otimes 2}_{\hbar}-\frac{1}{2}\left\vert b_{\hbar}\right\vert^{2} I\right)\ \dd x. \label{curl bh times bh} \end{align}
Substituting \eqref{eh times curl eh}-\eqref{curl bh times bh} inside \eqref{wedge quantum product term} we get 
\begin{align*}
\mathcal{T}_{5}&=\int_{\mathbb{R}^{3}}D_{x}u:\left(e^{\otimes 2}_{\hbar}-\frac{1}{2}\left\vert e_{\hbar}\right\vert^{2}I+b^{\otimes 2}_{\hbar}-\frac{1}{2}\left\vert b_{\hbar}\right\vert^{2}I  \right)\ \dd x\\
&-\int_{\mathbb{R}^{3}}u\cdot \left((\rho u -J_{\hbar})\times b_{\hbar}\right)\ \dd x
+\int_{\mathbb{R}^{3}}u\cdot e_{\hbar}\mathrm{div}e_{\hbar}\ \dd x. 
\end{align*}
Furthermore, using that $J\cdot (u\times b)=-u\cdot (J\times b)$ we see that  
\begin{align*}
\mathcal{T}_{3}=\int_{\mathbb{R}^{3}}J_{\hbar}\cdot (u\times b_{\hbar})\ \dd x=-\int_{\mathbb{R}^{3}}u\cdot ( J_{\hbar}\times b_{\hbar})\ \dd x . 
\end{align*}
Consequently, it follows that 
\begin{align}
\mathcal{T}_{3}+\mathcal{T}_{5}&=\int_{\mathbb{R}^{3}}D_{x}u:\left(e_{\hbar}^{\otimes 2}-\frac{1}{2}\left\vert e_{\hbar}\right\vert^{2}I+b_{\hbar}^{\otimes 2}-\frac{1}{2}\left\vert b_{\hbar}\right\vert^{2}I\right)\ \dd x \notag\\
&\underset{=0}{\underbrace{-\int_{\mathbb{R}^{3}}u\cdot ( J_{\hbar}\times b_{\hbar})\ \dd x +\int_{\mathbb{R}^{3}}u\cdot (J_{\hbar}\times b_{\hbar})\ \dd x}}+\int_{\mathbb{R}^{3}}u\cdot e_{\hbar}\mathrm{div}e_{\hbar}\ \dd x \notag\\
&=\int_{\mathbb{R}^{3}}D_{x}u:\left(e_{\hbar}^{\otimes 2}-\frac{1}{2}\left\vert e_{\hbar}\right\vert^{2}I+b_{\hbar}^{\otimes 2}-\frac{1}{2}\left\vert b_{\hbar}\right\vert^{2}I\right)\ \dd x+\int_{\mathbb{R}^{3}}u\cdot e_{\hbar}\mathrm{div}e_{\hbar}\ \dd x. \label{sum of quantum T3T5}     
\end{align}
Finally, using that $\mathrm{div}E=\rho$ and that $\mathrm{div}E_{\hbar}=\rho_{\hbar}$ we see that 
\begin{align}
\mathcal{T}_{2}&=-\int_{\mathbb{R}^{3}}(\rho_{\hbar}-\rho)u\cdot e_{\hbar}\ \dd x=-\int_{\mathbb{R}^{3}}u\cdot e_{\hbar}\mathrm{div}e_{\hbar}\ \dd x  \label{quantum T2 identity} 
\end{align}
and hence 
\begin{align}
\mathcal{T}_{2}+\mathcal{T}_{3}+\mathcal{T}_{5}=\int_{\mathbb{R}^{3}}D_{x}u:\left(e_{\hbar}^{\otimes 2}-\frac{1}{2}\left\vert e_{\hbar}\right\vert^{2}I+b_{\hbar}^{\otimes 2}-\frac{1}{2}\left\vert b_{\hbar}\right\vert^{2}I\right)\ \dd x.  \label{T2T3T5}   
\end{align}
Substituting \eqref{T2T3T5} inside 
\eqref{formula for time der of quantum G} we deduce that 
\begin{align}
\frac{\dd}{\dd t}\mathcal{G}_{\hbar}(t)&=-\frac{1}{4}\sum_{1\leq j,k\leq 3}\mathrm{tr}\left((i\hbar\partial_{x_{j}}+A^{j}_{\hbar}+u^{j})\vee \left((i\hbar\partial_{x_{k}}+A^{k}_{\hbar}+u^{k})\vee \partial_{x_{k}}u^{j}\right)R_{\hbar}\right) \notag\\
&+\int_{\mathbb{R}^{3}}D_{x}u:\left(e_{\hbar}^{\otimes 2}-\frac{1}{2}\left\vert e_{\hbar}\right\vert^{2}I+b_{\hbar}^{\otimes 2}-\frac{1}{2}\left\vert b_{\hbar}\right\vert^{2}I\right)\ \dd x \notag
\\
&-\int_{\mathbb{R}^{3}}\partial_{t}u\cdot (e_{\hbar}\times b_{\hbar})\ \dd x\coloneqq \sum_{k=1}^{3}I_{k},\label{right hand side of time der of Gh} 
\end{align}
which is the asserted formula \eqref{Time derivative of quantum corrected modulated energy statement}. 
\\
\textbf{Gr\"onwall estimate on $\mathcal{H}_{\hbar}(t)$.} We start by estimating $I_{1}$.  Set $P^{j}\coloneqq \Pi^{j}+u^{j}$. We have 
\begin{align}
P^{j}\vee (P
^{k}\vee \partial_{x_{k}}u^{j})&=P^{j}\vee (P^{k}\partial_{x_{k}}u^{j}+\partial_{x_{k}}u^{j}P^{k}) \notag\\
&=P^{j}P^{k}\partial_{x_{k}}u^{j}+P^{j}\partial_{x_{k}}u^{j}P^{k}+P^{k}\partial_{x_{k}}u^{j}P^{j}+\partial_{x_{k}}u^{j}P^{k}P^{j}. \label{formula for anticommutator of Pj}  \end{align}
Therefore we have 
\begin{align*}
I_{1}=&-\frac{1}{4}\sum_{1\leq j,k\leq 3}\langle \psi_{\hbar}\vert P^{j}P^{k}\partial_{x_{k}}u^{j}\vert \psi_{\hbar}\rangle -\frac{1}{4}\sum_{1\leq j,k\leq 3}\langle \psi_{\hbar} \vert  P^{j}(\partial_{x_{k}}u^{j}+\partial_{x_{j}}u^{k})P^{k} \vert \psi_{\hbar}\rangle\\
&-\frac{1}{4}\sum_{1\leq j,k\leq 3}\langle \psi_{\hbar} \vert  \partial_{x_{k}}u^{j}P^{k}P^{j}\vert  \psi_{\hbar}\rangle \coloneqq I_{1}^{1}+I_{1}^{2}+I_{1}^{3}.        
\end{align*}
To estimate $I_{1}^{1}$, we apply the self-adjointness of $P^{j}$ to find
\begin{align*}
\sum_{k=1}^{3}\langle \psi_{\hbar}\vert   P^{j}P^{k}\partial_{x_{k}}u^{j}\vert\psi_{\hbar}\rangle  &=\sum_{k=1}^{3}\langle P^{j}\psi_{\hbar}\vert   P^{k}\partial_{x_{k}}u^{j}\vert \psi_{\hbar} \rangle\\
&=\sum_{k=1}^{3} (\langle P^{j}\psi_{\hbar}  \vert  i\hbar \partial_{x_{k}x_{k}}u^{j}\vert \psi_{\hbar}\rangle +\langle P^{j}\psi_{\hbar}\vert  \partial_{x_{k}}u^{j}P^{k}\vert \psi_{\hbar}\rangle)\\
&=\langle P^{j}\psi_{\hbar}\vert  i\hbar \Delta u^{j}\vert \psi_{\hbar}\rangle+\sum_{k=1}^{3}\langle P^{j}\psi_{\hbar}  \vert \partial_{x_{k}}u^{j}P^{k}\vert \psi_{\hbar} \rangle.       
\end{align*}
Observe that
\begin{align*}
\left\vert \langle P^{j}\psi_{\hbar}\vert  \partial_{x_{k}}u^{j}P^{k}\vert \psi_{\hbar}\rangle \right\vert  \leq \frac{1}{2}\left\Vert D_{x}u\right\Vert_{\infty}\left(\left\Vert P^{j}\psi_{\hbar}(t,\cdot)\right\Vert_{2}^{2} +\left\Vert P^{k}\psi_{\hbar}(t,\cdot)\right\Vert_{2}^{2} \right).     
\end{align*}
Furthermore, since the $L^{2}$ norm of $\psi_{\hbar}$ is conserved ($\left\Vert \psi_{\hbar}(t,\cdot)\right\Vert_{2}=\left\Vert \psi_{\hbar}^{\mathrm{in}}\right\Vert_{2}=1$) we have 
\begin{align*}
\left\vert \langle P^{j}\psi_{\hbar} \vert i\hbar \Delta_{x}u^{j}\vert \psi_{\hbar}\rangle\right\vert  \leq \frac{1}{2}\left\Vert P^{j}\psi_{\hbar}(t,\cdot) \right\Vert_{2}^{2}+\frac{\hbar^{2}}{2}\left\Vert  \Delta_{x}u^{j}\psi_{\hbar}(t,\cdot)\right\Vert_{2}^{2}\leq \mathcal{U}_{\hbar}(t)+\frac{\hbar^{2}}{2}\left\Vert \Delta_{x} u\right\Vert_{\infty}^{2}.     
\end{align*}
Therefore, we obtain 
\begin{align}
\left\vert I_{1}^{1}\right\vert &\leq 3\left\Vert D_{x}u\right\Vert_{\infty}\mathcal{U}_{\hbar}(t)+9\mathcal{U}_{\hbar}(t)+9\hbar^{2} \left\Vert \Delta_{x}u\right\Vert_{\infty}^{2} \notag\\
&= (3\left\Vert D_{x}u\right\Vert_{\infty} +9)\mathcal{U}_{\hbar}(t)+9\hbar^{2}\left\Vert \Delta_{x}u \right\Vert_{\infty}^{2}.  \label{Est D1}    
\end{align}
To estimate $I_{1}^{2}$, we have 
\begin{align*}
\left\vert \langle \psi_{\hbar}\vert   P^{j}(\partial_{x_{k}}u^{j}+\partial_{x_{j}}u^{k})P^{k}\vert \psi_{\hbar}\rangle\right\vert &= \left\vert \langle P^{j}\psi_{\hbar}\vert  (\partial_{x_{k}}u^{j}+\partial_{x_{j}}u^{k})P^{k}\vert\psi_{\hbar}\rangle\right\vert \\
&\leq 2\left\Vert D_{x}u\right\Vert_{\infty} \left(\frac{1}{2}\left\Vert P^{j}\psi_{\hbar}(t,\cdot)\right\Vert_{2}^{2}+\frac{1}{2}\left\Vert P^{k}\psi_{\hbar}(t,\cdot)\right\Vert_{2}^{2}\right).      
\end{align*}
Therefore it follows that 
\begin{align}
\left\vert I_{1}^{2}\right\vert \leq 6\left\Vert D_{x}u\right\Vert_{\infty}\mathcal{U}_{\hbar}(t). \label{Est D2}     
\end{align}
To estimate $I_{1}^{3}$ we first note that 
\begin{align*}
&\sum_{k=1}^{3}\langle \psi_{\hbar} \vert  \partial_{x_{k}}u_{j}P^{k}P^{j}\vert \psi_{\hbar}\rangle=\sum_{k=1}^{3}\langle P^{k}\partial_{x_{k}}u^{j}\psi_{\hbar}\vert  P^{j}\vert \psi_{\hbar} \rangle\\
&=\sum_{k=1}^{3}\langle \partial_{x_{k}}u^{j}P^{k}\psi_{\hbar}\vert  P^{j}\vert \psi_{\hbar} \rangle  +\sum_{k=1}^{3}\langle i\hbar \partial_{x_{k}x_{k}}u^{j}\psi_{\hbar}\vert  P^{j}\vert \psi_{\hbar}\rangle \\
&= \sum_{k=1}^{3}\langle \partial_{x_{k}}u^{j}P^{k}\psi_{\hbar}\vert P^{j}\vert \psi_{\hbar}\rangle+\langle i\hbar \Delta u^{j}\psi_{\hbar} \vert P
^{j}\vert \psi_{\hbar}\rangle.    
\end{align*}
Therefore it follows that 
\begin{align}
\left\vert I_{1}^{3}\right\vert&\leq \frac{\left\Vert D_{x}u\right\Vert_{\infty}}{2}\sum_{1\leq j,k\leq 3} \left(\left\Vert P^{k}\psi_{\hbar}(t,\cdot)\right\Vert_{2}^{2} \notag 
+\left\Vert P^{j}\psi_{\hbar}(t,\cdot)\right\Vert_{2}^{2}\right)\\
&+\frac{\hbar^{2}}{2}\sum_{j=1}^{3}\left\Vert \Delta u^{j}\psi_{\hbar}(t,\cdot)\right\Vert_{2}^{2}+\frac{1}{2}\sum_{j=1}^{3}\left\Vert P^{j}\psi_{\hbar}\right\Vert_{2}^{2} \notag\\
&\leq 3(\left\Vert D_{x}u\right\Vert_{\infty}+\frac{1}{2})\mathcal{U}_{\hbar}(t)+\frac{9\hbar^{2}}{2}\left\Vert \Delta u\right\Vert_{\infty}^{2} . \label{EST D3}      
\end{align}
Gathering \eqref{Est D1}-\eqref{EST D3} we conclude that 
\begin{align}
I_{1}\leq c_{1}\mathcal{U}_{\hbar}(t)+c_{2}\hbar^{2},    \label{final est D1} 
\end{align}
where $c_{1}=c_{1}(\left\Vert D_{x}u\right\Vert_{\infty})$ and $c_{2}=c_{2}(\left\Vert \Delta_{x}u\right\Vert_{\infty})$. 
To estimate $I_{2}$ we note that 
\begin{align*}
\left\vert I_{2} \right\vert\leq& \left\vert \int_{\mathbb{R}^{3}}D_{x}u:e_{\hbar}^{\otimes 2}\ \dd x\right\vert+\frac{1}{2}\left\vert \int_{\mathbb{R}^{3}}D_{x}u: \left\vert e_{\hbar}\ \right\vert^{2}I\ \dd x \right\vert \\
&+\left\vert \int_{\mathbb{R}^{3}}D_{x}u:b_{\hbar}^{\otimes 2}\ \dd x\right\vert+\frac{1}{2}\left\vert \int_{\mathbb{R}^{3}}D_{x}u:\left\vert b_{\hbar}\right\vert^{2}I \ \dd x\right\vert \\
&\leq \left(\left\Vert D_{x}u\right\Vert_{\infty}+\frac{1}{2}\left\Vert \mathrm{div}u\right\Vert_{\infty}\right)\int_{\mathbb{R}^{3}}\left\vert e_{\hbar}\right\vert^{2}\ \dd x+\left(\left\Vert D_{x}u\right\Vert_{\infty}+\frac{1}{2}\left\Vert \mathrm{div}u\right\Vert_{\infty}  \right)\int_{\mathbb{R}^{3}}\left\vert b_{\hbar}\right\vert^{2}\ \dd x\\  &\leq (2\left\Vert D_{x}u\right\Vert_{\infty}+\left\Vert \mathrm{div}u\right\Vert_{\infty})(\mathcal{E}_{\hbar}(t)+\mathcal{B}_{\hbar}(t)).        
\end{align*}
As for $I_{3}$, we have 
\begin{align*}
\left\vert I_{3} \right\vert\leq \frac{1}{2}\left\Vert \partial_{t}u\right\Vert_{\infty}\int_{\mathbb{R}^{3}}(\left\vert e_{\hbar}\right\vert^{2}+\left\vert b_{\hbar}\right\vert^{2}) \ \dd x\leq \left\Vert \partial_{t}u\right\Vert_{\infty}\left(\mathcal{E}_{\hbar}(t)+\mathcal{B}_{\hbar}(t)\right).         
\end{align*}
Consequently, it follows that 
\begin{align*}
\frac{\dd}{\dd t}\mathcal{G}_{\hbar}(t)\leq C(\mathcal{H}_{\hbar}(t)+\hbar^{2})    
\end{align*}
for some $C=C(\left\Vert u\right\Vert_{W^{1,\infty}_{t,x}([0,T]\times \mathbb{R}^{3})},\left\Vert u\right\Vert_{L^{\infty}([0,T];W^{2,\infty}(\mathbb{R}^{3}))})$. 
Therefore integrating in time and owing to Lemma \ref{sandwich lemma} we conclude that 
\begin{align*}
(1-\left\Vert u\right\Vert_{\infty})\mathcal{H}_{\hbar}(t)&\leq \mathcal{G}_{\hbar}(t)\leq \mathcal{G}_{\hbar}(0)+C\left(\int_{0}^{t}\mathcal{H}_{\hbar}(\tau)\ \dd \tau+\hbar^{2}t\right)\\
&\leq (1+\left\Vert u^{\mathrm{in}}\right\Vert_{\infty})\mathcal{H}_{\hbar}(0) +C\hbar^{2}t+C\int_{0}^{t}\mathcal{H}_{\hbar}(\tau)\ \dd \tau   
\end{align*}
or equivalently 
\begin{align*}
\mathcal{H}_{\hbar}(t)\leq \frac{1+\left\Vert u^{\mathrm{in}}\right\Vert_{\infty}}{1-\left\Vert u\right\Vert_{\infty}}\mathcal{H}_{\hbar}(0)+\frac{C\hbar^{2}t}{1-\left\Vert u\right\Vert_{\infty}}+\frac{C}{1-\left\Vert u\right\Vert_{\infty}}\int_{0}^{t}\mathcal{H}_{\hbar}(\tau)\ \dd \tau.   \end{align*}
By Gr\"onwall's lemma we deduce 
\begin{align*}
\underset{t\in [0,T]}{\sup}\mathcal{H   }_{\hbar}(t)\leq Ce^{CT}(\mathcal{H}_{\hbar}(0)+\hbar^{2}T),    \end{align*}
which proves \eqref{vanishing of quantum modulated energy statement}. \\
\textbf{The convergence \eqref{convergence of E and B}-\eqref{convergence of rho and J}.} That 
\begin{align}
\underset{t\in [0,T]}{\sup}(\left\Vert (B_{\hbar}-B)(t,\cdot)\right\Vert_{2} +\left\Vert (E_{\hbar}-E)(t,\cdot)\right\Vert_{2})\underset{\hbar\rightarrow 0}{\rightarrow} 0 \label{convergence of Bh and Eh}    
\end{align}
follows immediately from \eqref{vanishing of quantum modulated energy statement}. Since $E_{\hbar}=-\partial_{t}A_{\hbar}-\nabla \varphi_{\hbar}$ and $\mathrm{div}A_{\hbar}=0$, taking the divergence we obtain 
\begin{align*}
\mathrm{div}E_{\hbar}=-\Delta\varphi_{\hbar}=\rho_{\hbar}.    
\end{align*}
Therefore, it follows that 
\begin{align*}
\left\Vert (\rho_{\hbar}-\rho)(t,\cdot)\right\Vert_{H^{-1}}^{2}&=\int_{\mathbb{R}^{3}}\frac{\left\vert \mathcal{F}(\mathrm{div}(E_{\hbar}-E))\right\vert^{2}(\xi)}{1+\left\vert \xi\right\vert^{2} }\ \dd \xi\\
&\leq\int_{\mathbb{R}^{3}}\frac{\left\vert \xi\right\vert^{2} }{1+\left\vert \xi\right\vert^{2}}\left\vert \mathcal{F}(E_{\hbar}-E)\right\vert^{2}(\xi)\ \dd \xi\leq \left\Vert (E_{\hbar}-E)(t,\cdot)\right\Vert_{2}^{2}      
\end{align*}
and hence in view \eqref{convergence of Bh and Eh} it follows that  
\begin{align}
\underset{t\in [0,T]}{\sup}\left\Vert (\rho_{\hbar}-\rho)(t,\cdot)\right\Vert_{H^{-1}}\underset{\hbar\rightarrow 0}{\rightarrow} 0. \label{convergence of rhohbar}     \end{align}
In addition, given a test vector field $\Psi\in W^{1,\infty}(\mathbb{R}^{3};\mathbb{R}^{3})$ with $\left\Vert \Psi\right\Vert_{W^{1,\infty}}\leq 1$ we have 
\begin{align*}
\int_{\mathbb{R}^{3}}(J_{\hbar}-\rho u)\cdot \Psi\ \dd x&=\int_{\mathbb{R}^{3}}(J_{\hbar}-\rho_{\hbar}u)\cdot \Psi\ \dd x+\int_{\mathbb{R}^{3}}(\rho_{\hbar}-\rho)u\cdot \Psi\ \dd x\\
&=-\frac{1}{2}\mathrm{tr}\left((i\hbar\nabla+A_{\hbar}+u)\vee \Psi R_{\hbar}\right)+\int_{\mathbb{R}^{3}}(\rho_{\hbar}-\rho) u\cdot \Psi\ \dd x
\\
&\leq  \sqrt{2\mathcal{U}_{\hbar}(t)}+\left\Vert (u\cdot \Psi)(t,\cdot)\right\Vert_{H^{1}} \left\Vert (\rho_{\hbar}-\rho)(t,\cdot)\right\Vert_{H^{-1}}.      
\end{align*}
Due to the convergence \eqref{convergence of rhohbar} and using that $\underset{t\in [0,T]}{\sup}\mathcal{U}_{\hbar}(t)\underset{\hbar\rightarrow 0}{\rightarrow} 0$ we find  $$\underset{t\in [0,T]}{\sup}\left\Vert (J_{\hbar}-\rho u)(t,\cdot)\right\Vert_{\mathrm{BL}} \underset{\hbar\rightarrow 0}{\rightarrow} 0. $$ 
\qed

\vspace{0.4 cm}
\begin{rem}
One may construct  admissible initial data, that is initial data such that $\mathcal{H}_{\hbar}(0)\underset{\hbar \rightarrow 0}{\rightarrow}0$. In more detail, for given $(\rho^{\mathrm{in}},u^{\mathrm{in}},E^{\mathrm{in}},B^{\mathrm{in}})$ satisfying \eqref{EMI} we may find  $(\psi_{\hbar}^{\mathrm{in}},A^{\mathrm{in}}_{0,\hbar},A^{\mathrm{in}}_{1,\hbar})$ satisfying \eqref{MS initial data} such that 
\begin{align*}
\mathcal{U}_{\hbar}(0)=\frac{1}{2}\int_{\mathbb{R}^{3}}\left\vert (i\hbar \nabla +A^{\mathrm{in}}_{0,\hbar}+u^{\mathrm{in}})\psi_{\hbar}^{\mathrm{in}}\right\vert^{2}\ \dd x\underset{\hbar \rightarrow 0}{\rightarrow}0    
\end{align*}
and 
\begin{align*}
\mathcal{E}_{\hbar}(0)=\frac{1}{2}\int_{\mathbb{R}^{3}}\left\vert \nabla\Delta^{-1}\left\vert \psi_{\hbar}^{\mathrm{in}}\right\vert^{2}-A_{1,\hbar}^{\mathrm{in}}-E^{\mathrm{in}} \right\vert^{2}\ \dd x\underset{\hbar \rightarrow 0}{\rightarrow} 0       
\end{align*}
and 
\begin{align*}
\mathcal{B}_{\hbar}(0)=\frac{1}{2}\int_{\mathbb{R}^{3}}\left\vert \mathrm{curl}A^{\mathrm{in}}_{0,\hbar}-B^{\mathrm{in}}\right\vert^{2}\ \dd x\underset{\hbar\rightarrow 0}{\rightarrow}0.    
\end{align*}
We further assume that $B^{\mathrm{in}}=\mathrm{curl}\widetilde{A_{0}}$ for some $\widetilde{A_{0}}\in H^{5}(\mathbb{R}^{3})$ with $\mathrm{div}\widetilde{A_{0}}=0$.  We may then choose $A^{\mathrm{in}}_{0,\hbar}\equiv \widetilde{A_{0}}$ so that $\mathcal{B}_{\hbar}(0)\equiv 0$. Moreover assume there exist some $\widetilde{A_{1}}\in H^{4}(\mathbb{R}^{3})$ with $\mathrm{div}\widetilde{A_{1}}=0$ such that $E^{\mathrm{in}}=\nabla \Delta^{-1}\rho^{\mathrm{in}}-\widetilde{A_{1}}$. So choosing $A_{1,\hbar}^{\mathrm{in}}=\widetilde{A_{1}}$ we get 
\begin{align*}
\mathcal{E}_{\hbar}(0)=\frac{1}{2}\int_{\mathbb{R}^{3}}\left\vert \nabla \Delta^{-1}(\left\vert\psi^{\mathrm{in}}_{\hbar}\right\vert^
{2}- \rho^{\mathrm{in}})\right\vert^{2}\ \dd x.      
\end{align*}
According to Proposition 6.1 in \cite{porat2026magnetized}, given $\rho^{\mathrm{in}}\in H^{3}(\mathbb{R}^{3})\cap \mathcal{P}(\mathbb{R}^{3})$ and $u
^{\mathrm{in}}\in H^{4}(\mathbb{R}^{3}), A^{\mathrm{in}}_{0}\in H^{4}(\mathbb{R}^{3})$ there exists a family of wave functions $\{\psi_{\hbar}^{\mathrm{in}}\}_{\hbar>0}$ such that 
\begin{align*}
\mathcal{H}_{\hbar}(0)=\mathcal{U}_{\hbar}(0)+\mathcal{E}_{\hbar}(0)  \underset{\hbar \rightarrow 0}{\rightarrow} 0.   
\end{align*}
Thus $(\psi_{\hbar}^{\mathrm{in}}, \widetilde{A_{0}},\widetilde{A_{1}})$ is the requested initial data. 
\end{rem}
\noindent{\bf Acknowledgments.}
This work was prepared during the author's stay at the University of Texas at Austin. The author is indebted to the University of Texas at Austin for financial support.

\vspace{0.4 cm}

\noindent{\bf Conflict of Interest:} The author declares that he has no conflict of interest.

\bigskip
\noindent{\bf Data availability:} Data sharing is not applicable to this article as no datasets were generated or analyzed during the current study.

\bigskip

\noindent{\bf Declaration on the Use of Generative AI.} The author acknowledges the use of ChatGPT as an auxiliary tool in exploring mathematical ideas and checking certain intermediate arguments and computations. All AI-generated suggestions were independently assessed and verified by the author. The mathematical results, proofs, organization, and conclusions of the paper are the author’s own, and the manuscript was written in its entirety by the author, who assumes full responsibility for its content.
\bibliographystyle{abbrv}
\bibliography{references}
\end{document}